\documentclass[12pt]{article}
\usepackage{amsmath, latexsym, amsfonts, amssymb, amsthm, amscd}
\usepackage[left=2cm, right=1.5cm, top=2cm, bottom=2.5cm]{geometry}
\usepackage{upquote}
\usepackage{color}
\usepackage{setspace}
\usepackage{bm}
\usepackage{t4phonet}
\usepackage{tipa}
\usepackage[labelfont=bf]{caption}
\usepackage[english]{babel}

\usepackage[T1]{fontenc}
\usepackage[utf8]{inputenc}
\usepackage{lmodern}
\usepackage{bbm}
\usepackage{dsfont}

\usepackage{hyperref}
\usepackage{bookmark}


\usepackage{pgf,tikz}
\usepackage{mathrsfs}
\usetikzlibrary{arrows}	
\tikzstyle{lien}=[->,>=stealth,rounded corners=2pt,thick]	\tikzset{individu/.style={draw,#1},individu/.default={}}

\usepackage{float}

\newtheorem{corollary}{Corollary}
\newtheorem{proposition}{Proposition}
\newtheorem{lemma}{Lemma}

\newtheorem{theorem}{Theorem}
\newtheorem{definition}{Definition}

\begin{document}
\title{On scaling limits of multitype Galton-Watson trees with possibly infinite variance}
\author{G. Berzunza\footnote{ {\sc Institut f\"ur Mathematik, Universit\"at Z\"urich, Winterthurerstrasse 190, CH-8057 Z\"urich, Switzerland;} e-mail: gabriel.berzunza@math.uzh.ch}}
\maketitle

\vspace{0.1in}

\begin{abstract} 
\noindent 
In this work, we study asymptotics of multitype Galton–Watson trees with finitely many 
types. We consider critical and irreducible offspring distributions such
that they belong to the domain of attraction of a stable law, where the stability indices 
may differ. We show that after a proper rescaling, their corresponding 
height process converges to the continuous-time height process associated 
with a strictly stable spectrally positive L\'evy process. This gives an analogue 
of a result obtained by  Miermont \cite{Gr} in the case of multitype Galton-Watson 
trees with finite covariance matrices of the offspring distribution. Our approach 
relies on a remarkable decomposition for multitype trees into monotype trees introduced 
in \cite{Gr}.
\end{abstract}

\bigskip 


\section{Introduction}

In the pioneer works \cite{Al1, Al2}, Aldous introduced the continuum random tree as 
the limit of rescaled Galton-Watson (GW) trees conditioned on the total progeny for
offspring distributions having finite variance. Specifically, 
he proved that their properly rescaled contour functions converge in distribution in the 
functional sense to the normalized Brownian excursion, which codes the continuum random tree
as the contour function does 
for discrete trees. This work has motivated the study of the convergence of other rescaled paths 
obtained from GW trees possibly with infinite variance, 
such as the Lukasiewicz path and the height process. Duquesne 
and Le Gall \cite{Du} obtained in full generality an unconditional version of Aldous' 
result. More precisely, they showed that the concatenation of rescaled
height processes (or rescaled contour functions) converges in distribution to the 
so-called continuous-time height process associated to a spectrally positive L\'evy 
process. In particular, when the offspring distribution belongs to the domain of attraction of a 
stable law of index $\alpha \in (1,2]$, Duquesne \cite{Du1} showed that the height processes
of GW trees conditioned on having $n$ vertices converge in distribution to the normalized
excursion of the continuous-time height process associated with a strictly stable
spectrally positive L\'evy process of index $\alpha$. 

The present work has been motivated by the following result of Miermont \cite{Gr}, 
which extends the previous ones on monotype GW trees to multitype GW trees. 
Recall that  multitype GW trees are a generalization of usual GW trees that
describe the genealogy of a population where individuals are differentiated 
by types that determine their offspring distribution. More 
precisely, Miermont establishes an unconditional version for 
the convergence of the rescaled height process
of critical multitype GW trees with finitely many types to the reflected Brownian motion, 
under the hypotheses that the offspring distribution is irreducible and has finite
covariance  matrix. Moreover, under an additional exponential moment 
assumption, he also established that conditionally on the number individuals of a
given type, the limit is given by the normalized Brownian excursion. More recently, 
de Raphelis \cite{Loic} has extended the unconditional result in \cite{Gr} 
for multitype GW trees with infinitely many types, under similar assumptions. 
Informally speaking, these results claim that multitype GW trees behave asymptotically in 
a similar way as the monotype ones, at least in the finite variance case. Therefore, this 
suggests that we should expect an analogous behavior for multitype GW trees that satisfy 
weaker hypotheses. 

Our main goal is to show an analogue result for critical multitype GW trees with finitely 
many types whose offspring distribution is still irreducible, but may have infinite 
variance. Specifically, we are interested in establishing scaling limits for 
their associated height processes, when the offspring distributions belong to the domain 
of attraction of a stable law where the stability indices may differ. 
This will lead us to modify and extend the results of Miermont in \cite{Gr}.


In the rest of the introduction, we will describe our setting more precisely and
give the exact definition of multitype GW trees. We then provide the main assumptions
on the offspring distribution in Section \ref{ass}. This will enable us to 
state our main results in Section \ref{main}. 

\subsection{Multitype plane trees and forests} 

We recall the standard formalism for family trees. Let $U$ be the set of all labels:
\begin{eqnarray*}
 U = \bigcup_{n=0}^{\infty} \mathbb{N}^{n},
\end{eqnarray*}

\noindent where $\mathbb{N} = \{1, 2, \dots \}$ and with the convention $\mathbb{N}^{0} = \{\varnothing\}$. An element of $U$ is a sequence
$u = u_{1}\cdots u_{j}$ of positive integers, and we call $|u| = j$ the length of 
$u$ (with the convention $|\varnothing| = 0$). If 
$u = u_{1}\cdots u_{j}$ and $v = v_{1}\cdots v_{k}$ belong to $U$, we write
$uv = u_{1}\cdots u_{j}v_{1}\cdots v_{k}$ for the concatenation of $u$ and $v$. In
particular, note that $u \varnothing = \varnothing u = u$. For $u \in U$ and
$A \subseteq U$, we let $u A = \{uv: v \in A\}$, and we say that $u$ is a prefix (or ancestor) of 
$v$ if $v \in u U$, in which case we write $u \vdash v$. Recall that the set $U$ comes
with a natural lexicographical order 
$\prec$, such that $u \prec v$ if and only if either $u \vdash v$, or 
$u = wu^{\prime}$, $v = w v^{\prime}$ with nonempty words $u^{\prime}, v^{\prime}$ such that
$u_{1}^{\prime} < v_{1}^{\prime}$. 

A rooted planar tree $\mathbf{t}$ is a finite subset of $U$ which satisfies the following
conditions:
\begin{enumerate}
 \item[I.] $\varnothing \in \mathbf{t}$, we called it the root of $\mathbf{t}$. 
 \item[II.] For $u \in U$ and $i \in \mathbb{N}$, if $ui \in \mathbf{t}$ then 
 $u \in \mathbf{t}$, and $uj \in \mathbf{t}$ for every $1 \leq j \leq i$.
\end{enumerate}

We let $\mathbb{T}$ be the set of all rooted planar trees. We call vertices (or individuals)
the elements of a tree $\mathbf{t} \in \mathbb{T}$, the length $|u|$ is called the 
height of $u \in \mathbf{t}$. We write $c_{\mathbf{t}}(u) = \max\{i \in \mathbb{Z}_{+}: ui \in \mathbf{t} \}$
for the number of children of $u$. The vertices of $\mathbf{t}$ with no children are called leaves.
For $\mathbf{t}$ a planar tree and $u \in \mathbf{t}$, we let $\mathbf{t}_{u} = \{v \in U:uv \in \mathbf{t} \}$
be the subtree of $\mathbf{t}$ rooted at $u$, which is itself a tree. The remaining part $[\mathbf{t}]_{u} = \{u\} \cup (\mathbf{t} \setminus u \mathbf{t}_{u})$
is called the subtree of $\mathbf{t}$ pruned at $u$. The lexicographical order $\prec$ will be
called the depth first order on $\mathbf{t}$.

In addition to trees, we are also interested in forest. A forest $\mathbf{f}$ is a  nonempty
subset of $U$ of the form
\begin{eqnarray*}
 \mathbf{f} = \bigcup_{k} k \mathbf{t}_{(k)},
\end{eqnarray*}

\noindent where $(\mathbf{t}_{(k)})$ is a finite or infinite sequence of trees, which
are called the components of $\mathbf{f}$. In words, a forest may be thought of as a rooted tree 
where the vertices at height one are the roots of the forest components. We let $\mathbb{F}$ be the set of rooted planar forests. For $\mathbf{f} \in \mathbb{F}$, we define the subtree 
$\mathbf{f}_{u} = \{v \in U: uv \in \mathbf{f} \} \in \mathbb{T}$ if $u \in \mathbf{f}$,
and $\mathbf{f}_{u} = \emptyset$ otherwise. Also, let 
$[\mathbf{f}]_{u} =\{ u \} \cup (\mathbf{f} \setminus u \mathbf{f}_{u}) \in \mathbb{F}$. 
With this notation, we observe that the tree components of $\mathbf{f}$ are 
$\mathbf{f}_{1}, \mathbf{f}_{2}, \dots$. We let $c_{\mathbf{f}}(u)$ be the number of
children of $u \in \mathbf{f}$. In particular, $c_{\mathbf{f}}(\varnothing) \in 
\mathbb{N} \cup \{\infty\}$ is the number of components of $\mathbf{f}$. We call 
$|u|-1$ the height of $u \in \mathbf{f}$. Notice that that notion of height differs from the
convention on trees because we want the roots of the forest components to be at height $0$. 

Let $d \in \mathbb{N}$, we call $[d] = \{1, \dots, d\}$ the set of types. A $d$-type 
planar tree, or simply a multitype tree is a pair $(\mathbf{t}, e_{\mathbf{t}})$, where 
$\mathbf{t} \in \mathbb{T}$ and $e_{\mathbf{t}}: \mathbf{t} \rightarrow [d]$ is a function
such that $e_{\mathbf{t}}(u)$ corresponds to the type of a vertex $u \in \mathbf{t}$. We let 
$\mathbb{T}^{(d)}$ be the set of $d$-type rooted planar trees. For $i \in [d]$, we write 
$c_{\mathbf{t}}^{(i)}(u) =
\max \{j \in \mathbb{Z}_{+}: uj \in \mathbf{t} \, \, \text{and} \, \, e_{\mathbf{t}}(uj) = i \}$ for
the number of offsprings of type $i$ of $u \in \mathbf{t}$. Then, 
$c_{\mathbf{t}}(u) = \sum_{i \in [d]} c_{\mathbf{t}}^{(i)}(u)$ is the total number of
children of $u \in \mathbf{t}$. Analogous definitions hold for $d$-type rooted planar forests
$(\mathbf{f}, e_{\mathbf{f}})$, whose set will be denoted by $\mathbb{F}^{(d)}$. 
For sake of simplicity, we shall frequently denote the type functions 
$e_{\mathbf{t}}$, $e_{\mathbf{f}}$ by $e$ when
it is free of ambiguity, and will even denote elements of $\mathbb{T}^{(d)}$, 
$\mathbb{F}^{(d)}$ by $\mathbf{t}$ or $\mathbf{f}$, without mentioning $e$. 
Moreover, it will be understood then that $\mathbf{t}_{u}$, $\mathbf{f}_{u}$, $[\mathbf{t}]_{u}$, $[\mathbf{f}]_{u}$ 
are marked with the appropriated function.

Finally, for $\mathbf{t} \in \mathbb{T}^{(d)}$ and $i \in [d]$, we let $\mathbf{t}^{(i)} = \{ u \in \mathbf{t}: e_{\mathbf{t}}(u) = i \}$
be the set of vertices on $\mathbf{t}$ bearing the type $i$, and $\mathbf{f}^{(i)}$ the 
corresponding notation for the forest $\mathbf{f} \in \mathbb{F}^{(d)}$.

\subsection{Multitype offspring distributions} \label{ass}
 
We set $\mathbb{Z}_{+} = \{ 0, 1, 2, \dots \}$ and $d \in \mathbb{N}$. A $d$-type
offspring distribution $\bm{\mu} = (\mu^{(1)}, \dots, \mu^{(d)})$
is a family of distributions on the space $\mathbb{Z}_{+}^{d}$ of integer-valued 
non-negative sequences of length $d$. It will be useful to introduce the 
Laplace transforms $\varphi = (\varphi^{(1)}, \dots, \varphi^{(d)})$ 
of $\bm{\mu}$ by
\begin{eqnarray*}
 \varphi^{(i)}(\mathbf{s}) = \sum_{\mathbf{z} \in \mathbb{Z}_{+}^{d}} \mu^{(i)}(\{ \mathbf{z} \}) \exp(- \langle \mathbf{z},  \mathbf{s}\rangle ), \hspace*{6mm} \text{for} \, \, i \in [d],
\end{eqnarray*}

\noindent where $\mathbf{s} = \mathbf(s_{1},\dots, s_{d}) \in \mathbb{R}_{+}^{d}$ and 
$\langle x, y \rangle$ is the usual scalar product of two vectors $x,y \in \mathbb{R}^{d}$. 
We let $\mathbf{0}$ be the vector of $\mathbb{R}_{+}^{d}$ with all components equal to $0$. Then, for $i,j \in [d]$, we define the quantity
\begin{eqnarray*}
 m_{ij} = -\frac{\partial \varphi^{(i)}}{\partial s_{j}}(\mathbf{0}) = \sum_{\mathbf{z} \in \mathbb{Z}_{+}^{d}} z_{j} \mu^{(i)}(\{ \mathbf{z} \})
\end{eqnarray*}

\noindent that corresponds to the mean number of children of type $j$, given by an individual of type
$i$. We let $\mathbf{M} := (m_{ij})_{i,j \in [d]}$ be the mean matrix of $\bm{\mu}$, and
$\mathbf{m}_{i} = (m_{i1}, \dots, m_{id}) \in \mathbb{R}_{+}^{d}$ be the mean vector of the
measure $\mu^{(i)}$. 



We say that a measure ${\bm \mu}$ on $\mathbb{Z}_{+}^{d}$ is non-degenerate, if there exists 
at least one $i \in [d]$ so that 
\begin{eqnarray*}
 \mu^{(i)} \left(\left\{ \mathbf{z} \in \mathbb{Z}_{+}^{d}: \sum_{j=1}^{d} z_{j} \neq 1 \right\} \right) > 0.
\end{eqnarray*}

\noindent The offspring distribution that we consider in this work are assumed 
to be non-degenerate in order to avoid cases which will lead to infinite linear trees. 
 
\begin{definition}
The mean matrix (or the offspring distribution $\bm{\mu}$) is called irreducible,
if for every $i,j \in [d]$, there is some $n \in \mathbb{N}$ so that $m_{ij}^{(n)} > 0$, where 
$m_{ij}^{(n)}$ is the $ij$-entry of the matrix $\mathbf{M}^{n}$. 
\end{definition}

Recall also that if $\mathbf{M}$ is irreducible, then according to Perron-Frobenius theorem, 
$\mathbf{M}$ admits a unique eigenvalue $\rho$ which is simple, positive and 
with maximal modulus. Furthermore, the corresponding right and left eigenvectors can be
chosen positive and  we call them $\mathbf{a} = (a_{1}, \dots, a_{d})$ and
$\mathbf{b} = (b_{1}, \dots, b_{d})$ respectively, and
normalize them such that $\langle \mathbf{a}, 1 \rangle = \langle \mathbf{a}, \mathbf{b} \rangle =1$; 
see Chapter V of \cite{At}.  We then say that $\bm{\mu}$
is sub-critical if $\rho < 1$, critical $\rho = 1$ and supercritical if $\rho >1$. 

\paragraph{Main assumptions.} Throughout this work, we consider an offspring 
distribution ${\bm \mu} = \left(\mu^{(1)}, \dots, \mu^{(d)}\right)$ on
$\mathbb{Z}_{+}^{d}$ satisfying the following conditions:
\begin{itemize}
\item[($\mathbf{H}_{1}$)] ${\bm \mu}$ is irreducible, non-degenerate and critical.

\item[($\mathbf{H}_{2}.\mathbf{1}$)] Let $\Delta$ be a nonempty subset of $[d]$.
For every $i \in \Delta$, there 
exists $\alpha_{i} \in (1,2]$ such that the Laplace transform of 
$\mu^{(i)}$ satisfies
\begin{eqnarray*}
 \psi^{(i)}(\mathbf{s}) : = - \log \varphi^{(i)}(\mathbf{s}) = \langle \mathbf{m}_{i}, \mathbf{s} \rangle + | \mathbf{s}|^{\alpha_{i}} \Theta^{(i)} \left( \mathbf{s}/| \mathbf{s}| \right) + o(| \mathbf{s}|^{\alpha_{i}}),
 \hspace*{5mm} \text{as} \hspace*{2mm} | \mathbf{s}| \downarrow 0,
\end{eqnarray*}

\noindent for $\mathbf{s} \in \mathbb{R}_{+}^{d}$ and where 
\begin{eqnarray*}
 \Theta^{(i)}(\mathbf{s}) = \int_{\mathbf{S}^{d}} | \langle \mathbf{s}, \mathbf{y} \rangle |^{\alpha_{i}} \lambda_{i}( {\rm d} \mathbf{y}),
\end{eqnarray*}

\noindent with $\lambda_{i}$ a finite Borel non-zero measure on 
$\mathbf{S}^{d} = \{\mathbf{y} \in \mathbb{R}^{d}: |\mathbf{y}| =1 \}$ such that 
for $\alpha_{i} \in (1,2)$, $\lambda_{i}$ has support in 
$\{\mathbf{y} \in \mathbb{R}^{d}_{+}: |\mathbf{y}| =1 \}$. We write $| \cdot |$ for the 
Euclidean norm. 

\item[($\mathbf{H}_{2}.\mathbf{2}$)] For $i \in 
[d] \setminus \Delta$, the Laplace transform of 
$\mu^{(i)}$ satisfies
\begin{eqnarray*}
 \psi^{(i)}(\mathbf{s}) : = - \log \varphi^{(i)}(\mathbf{s}) = \langle \mathbf{m}_{i}, \mathbf{s} \rangle + o(| \mathbf{s}|^{\alpha_{i}}),
 \hspace*{5mm} \text{as} \hspace*{2mm} | \mathbf{s}| \downarrow 0.
\end{eqnarray*}

\noindent where $\alpha_{i} = \min_{j \in \Delta} \alpha_{j}$. 
\end{itemize}

Let us comment on these assumptions:
\begin{enumerate}
\item We notice that criticality, hypothesis ($\mathbf{H}_{1}$), implies finiteness of all
coefficients of the mean matrix $\mathbf{M}$.

 \item For $i \in [d]$, we say that $\mu^{(i)}$ has finite variance when
\begin{eqnarray*}
 \frac{\partial^{2} \varphi^{(i)}}{\partial s_{j} \partial s_{k}}(\mathbf{0}) < \infty,
 \hspace*{6mm} \text{for} \, \, \, j,k \in [d].
\end{eqnarray*}

\noindent We then write $\mathbf{Q}^{(i)}$ for its covariance matrix. In particular, 
when $\mu^{(i)}$ satisfies the condition ($\mathbf{H}_{2}.\mathbf{1}$) with $\alpha_{i} = 2$, one can
easily verify that it possess finite variance and that  it does not have variance 
when $\alpha_{i} \in (1,2)$. This shows that our assumptions
on the offspring distribution are less restrictive than the ones made in \cite{Gr}, where the 
author assumes finitess on the covariance matrices. 

\item In the case when $\mu^{(i)}$ has finite variance, one can consider a measure
$\lambda_{i}$ on $\mathbf{S}^{d}$ such that 
\begin{eqnarray*}
 \Theta^{(i)}(\mathbf{s}) = \langle \mathbf{s}, \mathbf{Q}^{(i)} \mathbf{s} \rangle,
 \hspace*{6mm} \mathbf{s} \in \mathbb{R}^{d}_{+};
\end{eqnarray*}

see for example Section 2.4 of Samorodnitsky and Taqqu \cite{SaT}. 

 \item Let ${\bm \xi}_{1}, {\bm \xi}_{2}, \dots$ be a sequence of i.i.d. random variables
 on $\mathbb{Z}_{+}^{d}$ with common distribution $\mu^{(i)}$ satisfying ($\mathbf{H}_{2}.\mathbf{1}$).
 We observe that 
\begin{eqnarray} \label{conv2}
- \log \mathbb{E}\left[ \exp \left(  -\left \langle \frac{1}{n^{1/\alpha_{i}}}\sum_{k =1}^{n} \left( {\bm \xi}_{k} - \mathbf{m}_{i}\right), \mathbf{s} \right \rangle  \right) \right] \underset{n \rightarrow \infty}{\rightarrow} | \mathbf{s}|^{\alpha_{i}} \Theta^{(i)} \left( \mathbf{s}/| \mathbf{s}| \right), 
  \hspace*{6mm} \mathbf{s} \in \mathbb{R}^{d}_{+}, 
\end{eqnarray} 
 
\noindent Then, we conclude that
 \begin{eqnarray} \label{conv}
   \frac{1}{n^{1/\alpha_{i}}}\sum_{k =1}^{n} \left( {\bm \xi}_{k} - \mathbf{m}_{i} \right)
   \xrightarrow[n \rightarrow \infty]{d} \mathbf{Y}_{\alpha_{i}},
 \end{eqnarray}
 
 where the convergence is in distribution and $\mathbf{Y}_{\alpha_{i}}$ is a  
 $\alpha_{i}$-stable random vector in $\mathbb{R}^{d}_{+}$ which Laplace exponent satisfies
 \begin{eqnarray*} 
  \psi_{\mathbf{Y}_{\alpha_{i}}}(\mathbf{s}) =  | \mathbf{s}|^{\alpha_{i}} \Theta^{(i)} \left( \mathbf{s}/| \mathbf{s}| \right), 
  \hspace*{6mm} \mathbf{s} \in \mathbb{R}^{d}_{+}.
 \end{eqnarray*}
 
Sato’s book \cite{Sa} and \cite{SaT} are good references for background on 
multivariate stable distributions. On the other hand, we notice from (\ref{conv2}) that the
equation (\ref{conv}) is equivalent to the hypothesis 
($\mathbf{H}_{2}.\mathbf{1}$).                

\item We point out that in the monotype case, that is $d=1$, 
the condition ($\mathbf{H}_{2}.\mathbf{1}$) may be thought as the analogous assumption made in 
\cite{Du1} and \cite{Ko}, in order to get the convergence of the rescaled monotype GW tree to 
the continuum stable tree.

\item For $i \in [d] \setminus \Delta$, let $\mu^{(i)}$ be a measure that satisfies the
hypothesis ($\mathbf{H}_{2}.\mathbf{2}$). We can rewrite the expression of its Laplace
exponent in the following way
\begin{eqnarray*}
 \psi^{(i)}(\mathbf{s}) : = - \log \varphi^{(i)}(\mathbf{s}) = \langle \mathbf{m}_{i}, \mathbf{s} \rangle + | \mathbf{s}|^{\alpha_{i} } \Theta^{(i)} \left( \mathbf{s}/| \mathbf{s}| \right) + o(| \mathbf{s}|^{\alpha_{i}}),
 \hspace*{5mm} \text{as} \hspace*{2mm} | \mathbf{s}| \downarrow 0,
\end{eqnarray*}

\noindent for $\mathbf{s} \in \mathbb{R}_{+}^{d}$ and where 
\begin{eqnarray*}
 \Theta^{(i)}(\mathbf{s}) = \int_{\mathbf{S}^{d}} | \langle \mathbf{s}, \mathbf{y} \rangle |^{\alpha_{i}} \lambda_{i}( {\rm d} \mathbf{y}),
\end{eqnarray*}

\noindent with $\lambda_{i} \equiv 0$. Recall that $\alpha_{i} = \min_{j \in \Delta} \alpha_{j}$ for 
$i \in [d] \setminus \Delta$. This will be useful for the rest of the work. 
\end{enumerate}

Finally, let $\underline{\alpha} = \min_{i \in [d]} \alpha_{i}$ and $\bar{\lambda} = 
\sum_{i \in [d]} \mathds{1}_{\{\underline{\alpha} = \alpha_{i}  \}} a_{i} \lambda_{i}$.
We define
\begin{eqnarray*}
 \bar{c} = \left( \langle \mathbf{a}, {\bm \Theta}(\mathbf{b}) \rangle \right)^{1/\underline{\alpha}}= \left( \int_{\mathbf{S}^{d}} | \langle  \mathbf{b}, \mathbf{y} \rangle |^{\underline{\alpha}} \bar{\lambda}({\rm d} \mathbf{y}) \right)^{1/\underline{\alpha}},
\end{eqnarray*}

\noindent where ${\bm \Theta}(\mathbf{s}) = (\Theta^{(1)}(\mathbf{s})\mathds{1}_{\{\underline{\alpha} = \alpha_{1}\}},
\dots, \Theta^{(d)}(\mathbf{s})\mathds{1}_{\{\underline{\alpha} = \alpha_{d}\}}) \in \mathbb{R}_{+}^{d}$,
for $\mathbf{s} \in \mathbb{R}_{+}^{d}$. We notice that $\bar{c} \not \equiv 0$ due to ($\mathbf{H}_{2}.\mathbf{1}$).
This constant will play a role similar to the constant defined in equation (2) of \cite{Gr}, i.e.,
it corresponds to the total variance of the offspring distribution
${\bm \mu}$, when the covariance matrices are finite. 

\subsection{Multitype Galton-Watson trees and forests} \label{Law}

Let $\bm{\mu}$ be a $d$-type offspring distribution. We define the 
law $\mathbf{P}_{\bm{\mu}}^{(i)}$ (or simply $\mathbf{P}^{(i)}$) 
of a $d$-type GW tree (or multitype GW tree) rooted at a vertex of type $i \in [d]$ and with offspring
distribution ${\bm \mu}$ by
\begin{eqnarray*}
 \mathbf{P}^{(i)}\left( T = \mathbf{t} \right) = \prod_{u \in \mathbf{t}} \frac{c_{\mathbf{t}}^{(1)}(u)! \dots c_{\mathbf{t}}^{(d)}(u)! }{c_{\mathbf{t}}(u)!} \mu^{(e_{\mathbf{t}}(u))}\left( \left \{ c_{\mathbf{t}}^{(d)}(u), \dots, c_{\mathbf{t}}^{(d)}(u) \right \} \right), 
\end{eqnarray*}

\noindent where $T: \mathbb{T}^{(d)} \rightarrow \mathbb{T}^{(d)}$ is the identity map (see e.g., \cite{De},
or Miermont \cite{Gr} for a formal construction of a probability measure on $\mathbb{T}^{(d)}$). In particular, 
under the criticality assumption, ($\mathbf{H}_{1}$), the multitype GW trees with offspring distribution ${\bm \mu}$
are almost surely finite. Similarly, for $\mathbf{x} = (x_{1}, \dots, x_{r})$ a finite sequence with terms in $[d]$, we define
$\mathbf{P}_{\bm{\mu}}^{\mathbf{x}}$ (or simply $\mathbf{P}^{\mathbf{x}}$) the law of 
multitype GW forest with roots of type $\mathbf{x}$ and with offspring distribution ${\bm \mu}$
as the image measure of $\bigotimes_{j=1}^{r} \mathbf{P}^{(x_{j})}$ by the map 
\begin{eqnarray*}
 (\mathbf{t}_{(1)}, \dots, \mathbf{t}_{(r)}) \longmapsto \cup_{k=1}^{r} k \mathbf{t}_{(k)},
\end{eqnarray*}

\noindent i.e., it is the law that makes the identity map $F: \mathbb{F}^{(d)} 
\rightarrow \mathbb{F}^{(d)}$ the random forest whose trees components 
$F_{1}, \dots, F_{r}$ are independent with respective laws $\mathbf{P}^{(x_{1})}, \dots,\mathbf{P}^{(x_{d})}$. A similar definition holds for an infinite sequence $\mathbf{x} \in [d]^{\mathbb{N}}$.

We then say that a $\mathbb{F}^{(d)}$-value random variable $F$ is a multitype 
GW forest with offspring distribution $\bm{\mu}$ and roots of type $\mathbf{x}$
when it has law $\mathbf{P}^{\mathbf{x}}$. Similarly, a 
$\mathbb{T}^{(d)}$-value random variable $T$ with law $\mathbf{P}^{(i)}$ is a multitype 
GW tree with offspring distribution $\bm{\mu}$ and root of type $i \in [d]$.

\subsection{Main results} \label{main}

In this section, we state our main results on the asymptotic behavior of $d$-type 
GW trees with offspring distribution satisfying our main assumptions. In this direction, 
we first recall the definition of the discrete height process associated to a forest 
$\mathbf{f} \in \mathbb{F}$.  

 Let us denote by $\# \mathbf{f}$ the total progeny (or the total number of vertices) of $\mathbf{f}$. 
Let $1=u_{\mathbf{f}}(0) \prec  u_{\mathbf{f}}(1) \prec \dots \prec u_{\mathbf{f}}(\# \mathbf{f}-1)$
be the list of vertices of $\mathbf{f}$ in depth-first order. The height process 
$H^{\mathbf{f}} = (H_{n}^{\mathbf{f}}, n \geq 0)$ is defined by $H_{n}^{\mathbf{f}} = |u_{\mathbf{f}}(n)|-1$, for $0 \leq n < \# \mathbf{f}$, 
with the convention that $H_{n}^{\mathbf{f}} = 0$ for $n \geq \# \mathbf{f}$. Detailed description and properties of this object can be found for example in \cite{Du1}.

Let $Y^{(\underline{\alpha})} = (Y_{s}, s \geq 0)$ be a strictly stable spectrally positive
L\'evy process with index $\underline{\alpha} \in (1,2]$ with Laplace
exponent
\begin{eqnarray*}
 \mathbb{E}[\exp(-\lambda Y_{s})] = \exp(- s\lambda^{\underline{\alpha}}), 
\end{eqnarray*}

\noindent for $\lambda \in \mathbb{R}_{+}$. 

We can now state our main result.

\begin{theorem} \label{teo1}
 Let $F$ be a $d$-type GW forest distributed according to $\mathbf{P}^{\mathbf{x}}$,
for some arbitrary $\mathbf{x} \in [d]^{\mathbb{N}}$. Then, under $\mathbf{P}^{\mathbf{x}}$, 
the following convergence in distribution holds for the Skorohod topology on the
space $\mathbb{D}(\mathbb{R}_{+}, \mathbb{R})$ of right-continuous functions with 
left limits:
  \begin{eqnarray*}
   \left( \frac{1}{n^{1-1/\underline{\alpha}}} H^{F}_{\lfloor ns \rfloor}, s \geq 0 \right) 
   \xrightarrow[n \rightarrow \infty]{d} \left(\frac{1}{\bar{c}} H_{s}, s \geq 0 \right), 
  \end{eqnarray*}
  \noindent where $H$ stands for the continuous-time height process associated with the strictly
  stable spectrally positive L\'evy process $Y^{(\underline{\alpha})}$.
\end{theorem}

In particular, we notice that this result implies the convergence in law of the $d$-type GW forest properly
rescaled towards the stable forest of index $\underline{\alpha}$ for the Gromov-Hausdorff topology; see for example
Lemma 2.4 of \cite{Le1}. On the other hand, when $\underline{\alpha} = 2$, it is well-known that
$(H_{s}, s \geq 0)$ is proportional to the reflected Brownian motion. 
The notion of height process for spectrally positive L\'evy process
has been studied in great detail in \cite{Du}.

Next, for $n \geq 0$, we let $\Upsilon_{n}^{\mathbf{f}}$
be the first letter of $u_{\mathbf{f}}(n)$, with the convention that for $n \geq \# \mathbf{f}$,
it equals the number of components of $\mathbf{f}$. In words, $\Upsilon_{n}^{\mathbf{f}}$
is the index of the tree component to which $u_{\mathbf{f}}(n)$ belongs. 

\begin{theorem} \label{teo2}
For $i \in [d]$, let $F$ be a $d$-type GW forest distributed according to 
$\mathbf{P}^{\mathbf{i}}$, where $\mathbf{i} = (i, i, \dots)$. Then, under $\mathbf{P}^{\mathbf{i}}$,
we have the following convergence in distribution in $\mathbb{D}(\mathbb{R}_{+}, \mathbb{R})$:
  \begin{eqnarray*} 
   \left( \frac{1}{n^{1-1/\underline{\alpha}}} \Upsilon_{\lfloor ns \rfloor}^{F}, s \geq 0 \right) 
   \xrightarrow[n \rightarrow \infty]{d} \left(- \frac{\bar{c}}{b_{i}} I_{s}, s \geq 0 \right), 
  \end{eqnarray*}
  \noindent where $I_{s}$ is the infimum at time $s$ of the strictly
  stable spectrally positive L\'evy process $Y^{(\underline{\alpha})}$.
\end{theorem}




Let us explain our approach while we describe the organization for the rest of
the paper. We begin by exposing in Section \ref{subs1} the key ingredient, that is, 
a remarkable decomposition of $d$-type forests into monotype forests. The plan then
is to compare the corresponding height processes of the multitype GW forest and
the monotype GW forest, and show that they are close for the Skorohod topology. 
In this direction, we will need to control the shape of
large $d$-type GW forests. First, we establish in Section \ref{subs2}
sub-exponential tail bounds for the height and the number of tree components of $d$-type
GW forests that may be of independent interest. Secondly, we estimate in Section 
\ref{subs3} the asymptotic repartition of vertices of either type. To be a little
more precise,  Proposition \ref{pro4} provides a convergence of types theorem for
multitype GW trees, which extends Theorem 1 (iii) in \cite{Gr}, for the infinite
variance case. Roughly speaking, it shows that all types are homogeneously 
distributed in the limiting tree. We conclude with the proofs of 
Theorem \ref{teo1} and \ref{teo2} in Section \ref{pteo1} by pulling back the
known results of Duquesne and Le Gall \cite{Du} on the convergence of the 
rescaled height process of monotype GW forests to the multitype GW forest. Finally, 
in Section \ref{app}, we present two applications. The first one is an immediately 
consequence of Theorem \ref{teo1} and \ref{teo2} which provides information about
the maximal height of a vertex in a multitype GW tree. Our second application 
involves a particular multitype GW tree, known as alternating two-type GW 
tree which appears frequently in the study of random planar maps. We establish a 
conditioned version of Theorem \ref{teo1} for this special tree. 

The global structure of the proofs is close to that \cite{Gr}. Although we will try 
to make this work as self-contained as possible, we will often refer the reader to 
this paper when the proofs are readily adaptable, and will rather 
focus on the new technical ingredients. One difficulty arises from
the fact that we are assuming weaker assumptions on the offspring distribution
than in \cite{Gr}, we do not assume a finitess of the covariances
matrices of the offspring distributions and this forces us to improve some of 
Miermont's estimates. 

\section{Preliminary results}

Through this section unless we specify otherwise, we let 
$F$ be $d$-type GW forest with law $\mathbf{P}^{\mathbf{x}}$ where 
$\mathbf{x} \in [d]^{\mathbb{N}}$ and such that its
offspring distribution ${\bm \mu} = (\mu^{(1)}, \dots, \mu^{(d)})$ satisfies the main assumptions. 
More precisely, it is important to keep in mind that there is a nonempty subset $\Delta$ of $[d]$
such that the family of distributions $(\mu^{(i)})_{i \in \Delta}$ satisfy
($\mathbf{H}_{2}.\mathbf{1}$) while the remainder $(\mu^{(i)})_{i \in [d] \setminus \Delta}$
fulfills ($\mathbf{H}_{2}.\mathbf{2}$).

\subsection{Decomposition of multitype GW forests} \label{subs1}

In this section, we introduce the projection function $\Pi^{(i)}$ defined by 
Miermont in \cite{Gr} that goes from the set of $d$-types planar forests to the set of monotype planar forests.
Roughly speaking, the function $\Pi^{(i)}$ removes all the vertices of type different
from $i$ and then it connects the remaining vertices with their most recent common ancestor,
preserving the lexicographical order. More precisely, set a $d$-type forest 
$\mathbf{f} \in \mathbb{F}^{d}$ and let $v_{1} \prec v_{2} \prec \cdots$
be the vertices of $\mathbf{f}^{(i)}$ listed in depth-first order such that all ancestors of $v_{k}$ have types different
from $i$. They will be the roots of the new forest. We then build a forest
$\Pi^{(i)}(\mathbf{f}) = \mathbf{f}^{\prime}$ 
with as many tree components as there are elements in $\{ v_{1}, v_{2}, \dots \}$.
Recursively, starting from the set of roots $1,2, \dots$ of $\mathbf{f}^{\prime}$,
for each $u \in \mathbf{f}^{\prime}$, we let $v_{u1}, v_{u2}, \dots, v_{uk}$
be vertices of $(v_{u}\mathbf{f}_{v_{u}}) \setminus \{v_{u}\}$ arranged in
lexicographical order and such that:
\begin{itemize}
 \item[I.] They have type $i$, i.e.  $e_{\mathbf{f}}(v_{uj}) = i$ for $1 \leq j \leq k$,
 \item[II.] All their ancestors on $(v_{u}\mathbf{f}_{v_{u}}) \setminus \{v_{u}\}$ have
 types different from $i$ (if any). 
\end{itemize}

\noindent Then, we add the vertices $u1, \dots, uk$ to $\mathbf{f}^{\prime}$ as children
of $u$, and continue iteratively. See Figure \ref{fig1} for an example when $d=3$. 

\definecolor{sqsqsq}{rgb}{0.12549019607843137,0.12549019607843137,0.12549019607843137}
\definecolor{qqwuqq}{rgb}{0.,0.39215686274509803,0.}
\definecolor{qqqqff}{rgb}{0.,0.,1.}
\definecolor{ffqqqq}{rgb}{1.,0.,0.}

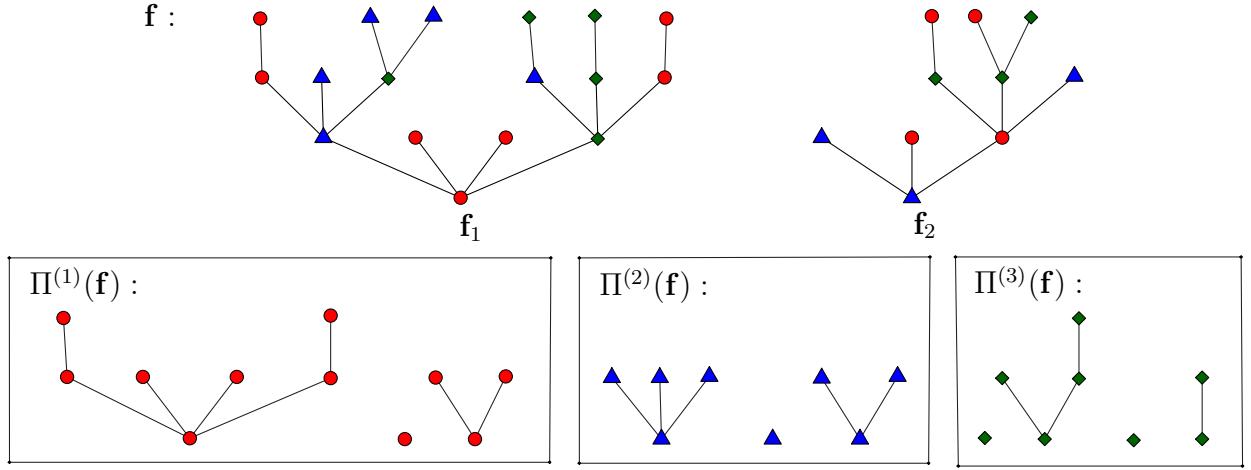
\begin{figure}[H]  \centering
\begin{minipage}[b]{10\linewidth}
\begin{tikzpicture}[line cap=round,line join=round,>=triangle 45,x=1.2cm,y=.8cm]
\clip(-2.5534599981273574,-4.354119095702462) rectangle (12.516868351254201,5.013922851210417);
\draw (0.8,3.)-- (1.48,2.);
\draw (1.46,3.)-- (1.48,2.);
\draw (2.2,2.98)-- (1.48,2.);
\draw (1.48,2.)-- (3.,1.);
\draw (2.5,2.)-- (3.,1.);
\draw (3.5,2.)-- (3.,1.);
\draw (3.,1.)-- (4.52,1.98);
\draw (4.5,2.98)-- (4.52,1.98);
\draw (3.82,3.)-- (4.52,1.98);
\draw (5.26,3.)-- (4.52,1.98);
\draw (7.,2.)-- (8.,1.);
\draw (8.,2.)-- (8.,1.);
\draw (9.,2.)-- (8.,1.);
\draw (8.26,2.98)-- (9.,2.);
\draw (9.,3.)-- (9.,2.);
\draw (9.8,3.02)-- (9.,2.);
\draw [color=sqsqsq] (-0.52,-1.98)-- (0.,-3.);
\draw [color=sqsqsq] (0.52,-1.98)-- (0.,-3.);
\draw [color=sqsqsq] (1.56,-2.)-- (0.,-3.);
\draw [color=sqsqsq] (8.22,4.02)-- (8.26,2.98);
\draw [color=sqsqsq] (8.7,4.02)-- (9.,3.);
\draw [color=sqsqsq] (9.32,4.)-- (9.,3.);
\draw [color=sqsqsq] (3.76,4.)-- (3.82,3.);
\draw [color=sqsqsq] (2.7,4.02)-- (2.2,2.98);
\draw [color=sqsqsq] (2.,4.)-- (2.2,2.98);
\draw [color=sqsqsq] (5.28,3.98)-- (5.26,3.);
\draw [color=sqsqsq] (0.78,3.98)-- (0.8,3.);
\draw [color=sqsqsq] (-1.4,-1.)-- (-1.36,-1.98);
\draw [color=sqsqsq] (-1.36,-1.98)-- (0.,-3.);
\draw [color=sqsqsq] (1.56,-0.96)-- (1.56,-2.);
\draw [color=sqsqsq] (2.722949349947181,-1.9906609483296136)-- (3.1589236097842797,-3.0142526888167174);
\draw [color=sqsqsq] (3.5001208566133135,-1.9717055457280006)-- (3.1589236097842797,-3.0142526888167174);
\draw [color=sqsqsq] (4.675355817913318,-1.9906609483296136)-- (5.225062493360094,-3.0142526888167174);
\draw [color=sqsqsq] (5.206107090758482,-1.9906609483296136)-- (5.225062493360094,-3.0142526888167174);
\draw [color=sqsqsq] (5.755813766205258,-1.9717055457280006)-- (5.225062493360094,-3.0142526888167174);
\draw [color=sqsqsq] (7.,-2.)-- (7.423889195147201,-3.0142526888167174);
\draw [color=sqsqsq] (7.840908052382686,-1.9717055457280006)-- (7.423889195147201,-3.0142526888167174);
\draw [color=sqsqsq] (9.,-2.)-- (9.471072676121404,-3.0142526888167174);
\draw [color=sqsqsq] (9.850180728153664,-2.0096163509312266)-- (9.471072676121404,-3.0142526888167174);
\draw [color=sqsqsq] (9.850180728153664,-1.004980013045736)-- (9.850180728153664,-2.0096163509312266);
\draw [color=sqsqsq] (11.214969715469797,-1.9906609483296136)-- (11.214969715469797,-3.0142526888167174);
\draw [color=sqsqsq] (-2.,0.)-- (4.,0.);
\draw [color=sqsqsq] (4.,0.)-- (3.9862251068017778,-3.3999403020825816);
\draw [color=sqsqsq] (3.9862251068017778,-3.3999403020825816)-- (-1.998042491393394,-3.3999403020825816);
\draw [color=sqsqsq] (-1.998042491393394,-3.3999403020825816)-- (-2.,0.);
\draw [color=sqsqsq] (4.315203168482665,-5.003313800800747E-4)-- (4.315203168482665,-3.4156059240673837);
\draw [color=sqsqsq] (4.315203168482665,-3.4156059240673837)-- (8.18461179872928,-3.4156059240673837);
\draw [color=sqsqsq] (8.18461179872928,-3.4156059240673837)-- (8.200277420714086,0.01516529060472407);
\draw [color=sqsqsq] (8.200277420714086,0.01516529060472407)-- (4.315203168482665,-5.003313800800747E-4);
\draw [color=sqsqsq] (8.482258616440559,0.015165290604728524)-- (8.513589860410168,-3.462602790021792);
\draw [color=sqsqsq] (8.513589860410168,-3.462602790021792)-- (11.662379879355795,-3.462602790021792);
\draw [color=sqsqsq] (11.662379879355795,-3.462602790021792)-- (11.64671425737099,0.015165290604728524);
\draw [color=sqsqsq] (11.64671425737099,0.015165290604728524)-- (8.482258616440559,0.015165290604728524);
\draw (7.8955098657369875,0.9565267571461435) node[anchor=north west] {$\mathbf{f}_{2}$};
\draw (2.8668452086148664,0.9095298911917311) node[anchor=north west] {$\mathbf{f}_{1}$};
\draw (-1.8955038747655841,0.03225506004269895) node[anchor=north west] {$\Pi^{(1)}(\mathbf{f}):$};
\draw (4.417741785110475,0.0165894380578948) node[anchor=north west] {$\Pi^{(2)}(\mathbf{f}):$};
\draw (8.569131611083565,0.047920682027503096) node[anchor=north west] {$\Pi^{(3)}(\mathbf{f}):$};
\draw (-0.6265884939964509,4.402963593803055) node[anchor=north west] {$\mathbf{f}:$};
\draw [color=sqsqsq] (4.48752501031551,4.025564518714591)-- (4.5,2.98);
\begin{scriptsize}
\draw [fill=ffqqqq] (3.,1.) circle (2.5pt);
\draw [fill=ffqqqq] (2.5,2.) circle (2.5pt);
\draw [fill=qqqqff,shift={(1.48,2.)}] (0,0) ++(0 pt,3.75pt) -- ++(3.2475952641916446pt,-5.625pt)--++(-6.495190528383289pt,0 pt) -- ++(3.2475952641916446pt,5.625pt);
\draw [fill=qqwuqq] (4.52,1.98) ++(-2.5pt,0 pt) -- ++(2.5pt,2.5pt)--++(2.5pt,-2.5pt)--++(-2.5pt,-2.5pt)--++(-2.5pt,2.5pt);
\draw [fill=ffqqqq] (3.5,2.) circle (2.5pt);
\draw [fill=ffqqqq] (5.26,3.) circle (2.5pt);
\draw [fill=qqwuqq] (4.5,2.98) ++(-2.5pt,0 pt) -- ++(2.5pt,2.5pt)--++(2.5pt,-2.5pt)--++(-2.5pt,-2.5pt)--++(-2.5pt,2.5pt);
\draw [fill=qqqqff,shift={(3.82,3.)}] (0,0) ++(0 pt,3.75pt) -- ++(3.2475952641916446pt,-5.625pt)--++(-6.495190528383289pt,0 pt) -- ++(3.2475952641916446pt,5.625pt);
\draw [fill=qqwuqq] (2.2,2.98) ++(-2.5pt,0 pt) -- ++(2.5pt,2.5pt)--++(2.5pt,-2.5pt)--++(-2.5pt,-2.5pt)--++(-2.5pt,2.5pt);
\draw [fill=ffqqqq] (0.8,3.) circle (2.5pt);
\draw [fill=qqqqff,shift={(1.46,3.)}] (0,0) ++(0 pt,3.75pt) -- ++(3.2475952641916446pt,-5.625pt)--++(-6.495190528383289pt,0 pt) -- ++(3.2475952641916446pt,5.625pt);
\draw [fill=qqqqff,shift={(8.,1.)}] (0,0) ++(0 pt,3.75pt) -- ++(3.2475952641916446pt,-5.625pt)--++(-6.495190528383289pt,0 pt) -- ++(3.2475952641916446pt,5.625pt);
\draw [fill=ffqqqq] (9.,2.) circle (2.5pt);
\draw [fill=ffqqqq] (8.,2.) circle (2.5pt);
\draw [fill=qqqqff,shift={(7.,2.)}] (0,0) ++(0 pt,3.75pt) -- ++(3.2475952641916446pt,-5.625pt)--++(-6.495190528383289pt,0 pt) -- ++(3.2475952641916446pt,5.625pt);
\draw [fill=qqwuqq] (8.26,2.98) ++(-2.5pt,0 pt) -- ++(2.5pt,2.5pt)--++(2.5pt,-2.5pt)--++(-2.5pt,-2.5pt)--++(-2.5pt,2.5pt);
\draw [fill=qqwuqq] (9.,3.) ++(-2.5pt,0 pt) -- ++(2.5pt,2.5pt)--++(2.5pt,-2.5pt)--++(-2.5pt,-2.5pt)--++(-2.5pt,2.5pt);
\draw [fill=qqqqff,shift={(9.8,3.02)}] (0,0) ++(0 pt,3.75pt) -- ++(3.2475952641916446pt,-5.625pt)--++(-6.495190528383289pt,0 pt) -- ++(3.2475952641916446pt,5.625pt);
\draw [fill=ffqqqq] (0.,-3.) circle (2.5pt);
\draw [fill=ffqqqq] (-0.52,-1.98) circle (2.5pt);
\draw [fill=ffqqqq] (0.52,-1.98) circle (2.5pt);
\draw [fill=ffqqqq] (1.56,-2.) circle (2.5pt);
\draw [fill=ffqqqq] (5.28,3.98) circle (2.5pt);
\draw [fill=ffqqqq] (-1.36,-1.98) circle (2.5pt);
\draw [fill=qqqqff,shift={(2.,4.)}] (0,0) ++(0 pt,3.75pt) -- ++(3.2475952641916446pt,-5.625pt)--++(-6.495190528383289pt,0 pt) -- ++(3.2475952641916446pt,5.625pt);
\draw [fill=qqqqff,shift={(2.7,4.02)}] (0,0) ++(0 pt,3.75pt) -- ++(3.2475952641916446pt,-5.625pt)--++(-6.495190528383289pt,0 pt) -- ++(3.2475952641916446pt,5.625pt);
\draw [fill=qqwuqq] (3.76,4.) ++(-2.5pt,0 pt) -- ++(2.5pt,2.5pt)--++(2.5pt,-2.5pt)--++(-2.5pt,-2.5pt)--++(-2.5pt,2.5pt);
\draw [fill=ffqqqq] (8.22,4.02) circle (2.5pt);
\draw [fill=ffqqqq] (8.7,4.02) circle (2.5pt);
\draw [fill=qqwuqq] (9.32,4.) ++(-2.5pt,0 pt) -- ++(2.5pt,2.5pt)--++(2.5pt,-2.5pt)--++(-2.5pt,-2.5pt)--++(-2.5pt,2.5pt);
\draw [fill=ffqqqq] (-1.4,-1.) circle (2.5pt);
\draw [fill=ffqqqq] (1.56,-0.96) circle (2.5pt);
\draw [fill=ffqqqq] (0.78,3.98) circle (2.5pt);
\draw [fill=ffqqqq] (2.3817521031181474,-3.0142526888167174) circle (2.5pt);
\draw [fill=ffqqqq] (3.1589236097842797,-3.0142526888167174) circle (2.5pt);
\draw [fill=ffqqqq] (2.722949349947181,-1.9906609483296136) circle (2.5pt);
\draw [fill=ffqqqq] (3.5001208566133135,-1.9717055457280006) circle (2.5pt);
\draw [fill=qqqqff,shift={(5.225062493360094,-3.0142526888167174)}] (0,0) ++(0 pt,3.75pt) -- ++(3.2475952641916446pt,-5.625pt)--++(-6.495190528383289pt,0 pt) -- ++(3.2475952641916446pt,5.625pt);
\draw [fill=qqqqff,shift={(4.675355817913318,-1.9906609483296136)}] (0,0) ++(0 pt,3.75pt) -- ++(3.2475952641916446pt,-5.625pt)--++(-6.495190528383289pt,0 pt) -- ++(3.2475952641916446pt,5.625pt);
\draw [fill=qqqqff,shift={(5.206107090758482,-1.9906609483296136)}] (0,0) ++(0 pt,3.75pt) -- ++(3.2475952641916446pt,-5.625pt)--++(-6.495190528383289pt,0 pt) -- ++(3.2475952641916446pt,5.625pt);
\draw [fill=qqqqff,shift={(5.755813766205258,-1.9717055457280006)}] (0,0) ++(0 pt,3.75pt) -- ++(3.2475952641916446pt,-5.625pt)--++(-6.495190528383289pt,0 pt) -- ++(3.2475952641916446pt,5.625pt);
\draw [fill=qqqqff,shift={(6.4571636624649384,-3.0142526888167174)}] (0,0) ++(0 pt,3.75pt) -- ++(3.2475952641916446pt,-5.625pt)--++(-6.495190528383289pt,0 pt) -- ++(3.2475952641916446pt,5.625pt);
\draw [fill=qqqqff,shift={(7.423889195147201,-3.0142526888167174)}] (0,0) ++(0 pt,3.75pt) -- ++(3.2475952641916446pt,-5.625pt)--++(-6.495190528383289pt,0 pt) -- ++(3.2475952641916446pt,5.625pt);
\draw [fill=qqqqff,shift={(7.,-2.)}] (0,0) ++(0 pt,3.75pt) -- ++(3.2475952641916446pt,-5.625pt)--++(-6.495190528383289pt,0 pt) -- ++(3.2475952641916446pt,5.625pt);
\draw [fill=qqqqff,shift={(7.840908052382686,-1.9717055457280006)}] (0,0) ++(0 pt,3.75pt) -- ++(3.2475952641916446pt,-5.625pt)--++(-6.495190528383289pt,0 pt) -- ++(3.2475952641916446pt,5.625pt);
\draw [fill=qqwuqq] (8.807633585064949,-2.9952972862151044) ++(-2.5pt,0 pt) -- ++(2.5pt,2.5pt)--++(2.5pt,-2.5pt)--++(-2.5pt,-2.5pt)--++(-2.5pt,2.5pt);
\draw [fill=qqwuqq] (9.471072676121404,-3.0142526888167174) ++(-2.5pt,0 pt) -- ++(2.5pt,2.5pt)--++(2.5pt,-2.5pt)--++(-2.5pt,-2.5pt)--++(-2.5pt,2.5pt);
\draw [fill=qqwuqq] (9.,-2.) ++(-2.5pt,0 pt) -- ++(2.5pt,2.5pt)--++(2.5pt,-2.5pt)--++(-2.5pt,-2.5pt)--++(-2.5pt,2.5pt);
\draw [fill=qqwuqq] (9.850180728153664,-2.0096163509312266) ++(-2.5pt,0 pt) -- ++(2.5pt,2.5pt)--++(2.5pt,-2.5pt)--++(-2.5pt,-2.5pt)--++(-2.5pt,2.5pt);
\draw [fill=qqwuqq] (10.456753611405277,-3.0332080914183304) ++(-2.5pt,0 pt) -- ++(2.5pt,2.5pt)--++(2.5pt,-2.5pt)--++(-2.5pt,-2.5pt)--++(-2.5pt,2.5pt);
\draw [fill=qqwuqq] (11.214969715469797,-3.0142526888167174) ++(-2.5pt,0 pt) -- ++(2.5pt,2.5pt)--++(2.5pt,-2.5pt)--++(-2.5pt,-2.5pt)--++(-2.5pt,2.5pt);
\draw [fill=qqwuqq] (11.214969715469797,-1.9906609483296136) ++(-2.5pt,0 pt) -- ++(2.5pt,2.5pt)--++(2.5pt,-2.5pt)--++(-2.5pt,-2.5pt)--++(-2.5pt,2.5pt);
\draw [fill=qqwuqq] (4.48752501031551,4.025564518714591) ++(-2.5pt,0 pt) -- ++(2.5pt,2.5pt)--++(2.5pt,-2.5pt)--++(-2.5pt,-2.5pt)--++(-2.5pt,2.5pt);
\draw [fill=qqwuqq] (9.850180728153664,-1.004980013045736) ++(-2.5pt,0 pt) -- ++(2.5pt,2.5pt)--++(2.5pt,-2.5pt)--++(-2.5pt,-2.5pt)--++(-2.5pt,2.5pt);
\draw [fill=black] (-2.,0.) ++(-0.5pt,0 pt) -- ++(0.5pt,0.5pt)--++(0.5pt,-0.5pt)--++(-0.5pt,-0.5pt)--++(-0.5pt,0.5pt);
\draw [fill=black] (4.,0.) ++(-0.5pt,0 pt) -- ++(0.5pt,0.5pt)--++(0.5pt,-0.5pt)--++(-0.5pt,-0.5pt)--++(-0.5pt,0.5pt);
\draw [fill=black] (3.9862251068017778,-3.3999403020825816) ++(-0.5pt,0 pt) -- ++(0.5pt,0.5pt)--++(0.5pt,-0.5pt)--++(-0.5pt,-0.5pt)--++(-0.5pt,0.5pt);
\draw [fill=black] (-1.998042491393394,-3.3999403020825816) ++(-0.5pt,0 pt) -- ++(0.5pt,0.5pt)--++(0.5pt,-0.5pt)--++(-0.5pt,-0.5pt)--++(-0.5pt,0.5pt);
\draw [fill=black] (4.315203168482665,-5.003313800800747E-4) ++(-0.5pt,0 pt) -- ++(0.5pt,0.5pt)--++(0.5pt,-0.5pt)--++(-0.5pt,-0.5pt)--++(-0.5pt,0.5pt);
\draw [fill=black] (4.315203168482665,-3.4156059240673837) ++(-0.5pt,0 pt) -- ++(0.5pt,0.5pt)--++(0.5pt,-0.5pt)--++(-0.5pt,-0.5pt)--++(-0.5pt,0.5pt);
\draw [fill=black] (8.18461179872928,-3.4156059240673837) ++(-0.5pt,0 pt) -- ++(0.5pt,0.5pt)--++(0.5pt,-0.5pt)--++(-0.5pt,-0.5pt)--++(-0.5pt,0.5pt);
\draw [fill=black] (8.200277420714086,0.01516529060472407) ++(-0.5pt,0 pt) -- ++(0.5pt,0.5pt)--++(0.5pt,-0.5pt)--++(-0.5pt,-0.5pt)--++(-0.5pt,0.5pt);
\draw [fill=black] (8.482258616440559,0.015165290604728524) ++(-0.5pt,0 pt) -- ++(0.5pt,0.5pt)--++(0.5pt,-0.5pt)--++(-0.5pt,-0.5pt)--++(-0.5pt,0.5pt);
\draw [fill=black] (8.513589860410168,-3.462602790021792) ++(-0.5pt,0 pt) -- ++(0.5pt,0.5pt)--++(0.5pt,-0.5pt)--++(-0.5pt,-0.5pt)--++(-0.5pt,0.5pt);
\draw [fill=black] (11.662379879355795,-3.462602790021792) ++(-0.5pt,0 pt) -- ++(0.5pt,0.5pt)--++(0.5pt,-0.5pt)--++(-0.5pt,-0.5pt)--++(-0.5pt,0.5pt);
\draw [fill=qqwuqq] (11.64671425737099,0.015165290604728524) ++(-0.5pt,0 pt) -- ++(0.5pt,0.5pt)--++(0.5pt,-0.5pt)--++(-0.5pt,-0.5pt)--++(-0.5pt,0.5pt);
\end{scriptsize}
\end{tikzpicture}
\end{minipage}
\caption{{\small A realization of the projection $\Pi^{(i)}$ for a
three-type planar forest with two tree components, type $1$ vertices 
represented with circles, type $2$ vertices with triangles and type $3$ vertices 
with diamonds.}} \label{fig1}
\end{figure}



We have the following key result:

\begin{proposition} \label{pro1}
 Let $\mathbf{x} \in [d]^{\mathbb{N}}$. Then, under the law $\mathbf{P}^{\mathbf{x}}$, 
 the forest $\Pi^{(i)}(F)$ is a monotype GW forest with critical non-degenerate offspring 
 distribution $\bar{\mu}^{(i)}$ that is in the domain of attraction of a stable 
 law of index $\underline{\alpha} = \min_{j \in [d]} \alpha_{j}$. More precisely, the Laplace
 exponent of $\bar{\mu}^{(i)}$ satisfies
 \begin{eqnarray*}
  \bar{\psi}^{(i)}(s) = s +
\frac{1}{a_{i}} \left(\frac{\bar{c}}{b_{i}} s \right)^{\underline{\alpha}}  
  + o(s^{\underline{\alpha}}),
  \hspace*{6mm} s \downarrow 0,
 \end{eqnarray*}
 
 \noindent where $s \in \mathbb{R}_{+}$. 
 \end{proposition}
 

The proof of this proposition is based in an inductive argument
that consists in removing types one by one until we are left with a 
monotype GW forests. More precisely, we suppose that the 
vertices with type $d$ are removed from the forest 
$\mathbf{f} \in \mathbb{F}^{(d)}$. We point out that one can 
delete any other type similarly. We let 
$v_{1} \prec v_{2} \prec \dots$ be the vertices of $\mathbf{f}$ listed
in depth-first order such that $e_{\mathbf{f}}(v_{i}) \neq d$ and 
$e_{\mathbf{f}}(v) = d$ for every $v \vdash v_{i}$. These are the vertices
of $\mathbf{f}$ with type different from $d$ which does not have ancestors of type $d$.
We build a forest $\tilde{\Pi}(\mathbf{f}) = \tilde{\mathbf{f}}$ recursively. We start
from the set $\{ v_{1}, v_{2}, \dots \}$ and for each $v_{u} \in \tilde{\mathbf{f}}$, 
we let $v_{u1} \prec \dots \prec v_{uk}$ be the descendants of $v_{u}$ in 
$\mathbf{f}$ such that:
\begin{itemize}
 \item[I.] They have type different from $d$.
 \item[II.] For $1 \leq j \leq k$, all the vertices between $v_{u}$ and $v_{uj}$ have
 type $d$ (if any). 
\end{itemize}

\noindent Then, we add these vertices to $\tilde{\mathbf{f}}$, and continue in an obvious 
way. We naturally associated the type $e_{\mathbf{f}}$ to the vertices
of $\tilde{\Pi}(\mathbf{f})$. In the sequel, we refer to this procedure as the 
{\bf $\mathbf{d}$- to $(\mathbf{d-1})$-type operation}. 

The following lemma shows that after performing the 
$d$-to $(d-1)$-type operation in the multitype GW
forest $F$, we obtain a $(d-1)$-type GW forest which 
offspring distribution still satisfying our 
main assumptions. First, we fix some notation.
We denote by $\tilde{\mathbf{m}}_{d}$ the 
vector in $\mathbb{R}_{+}^{d-1}$ with entries
\begin{eqnarray*}
\tilde{m}_{dk} = \frac{m_{dk}}{1-m_{dd}}, \hspace*{6mm} \text{for} \, \, k \in [d-1], 
\end{eqnarray*}

\noindent and for $j \in [d-1]$, we write $\tilde{\mathbf{m}}_{j}$ for the
vector in $\mathbb{R}_{+}^{d-1}$ with entries
 \begin{eqnarray*}
  \tilde{m}_{jk} = m_{jk} + \frac{m_{jd}m_{dk}}{1-m_{dd}}, \hspace*{6mm} \text{for} \, \, k \in [d-1]. 
 \end{eqnarray*}
 
 \noindent We stress that due to the irreducibility assumption on the mean matrix 
 $\mathbf{M}$ of the measure ${\bm \mu }$, we have that $1-\mu_{jj} > 0$ for all $j \in [d]$. Thus, all the previous quantities are finite.
 

\begin{lemma} \label{lemma1}
 Let $\mathbf{x} \in [d]^{\mathbb{N}}$. Then, under the law $\mathbf{P}^{\mathbf{x}}$,
 the forest $\tilde{\Pi}(F)$ is a non-degenerate, irreducible, critical $(d-1)$-type GW  
 forest. Moreover, its offspring distribution $\tilde{\bm{\mu}} = (\tilde{\mu}^{(1)}, \dots, \tilde{\mu}^{(d-1)})$
 has Laplace exponents
    \begin{eqnarray*} 
 \tilde{\psi}^{(j)}(\mathbf{s}) = \langle \tilde{\mathbf{m}}_{j}, \mathbf{s} \rangle  +
 |\mathbf{s}|^{\tilde{\alpha}_{j}} \tilde{\Theta}^{(j)} \left( \mathbf{s}/ |\mathbf{s}| \right)+o(|\mathbf{s}|^{\tilde{\alpha}_{j}}),
\hspace*{6mm} |\mathbf{s}| \downarrow 0,
  \end{eqnarray*}
 
  \noindent for $j \in [d-1]$, $\mathbf{s} \in \mathbb{R}_{+}^{d-1}$,
  $\tilde{\alpha}_{j} = \min(\alpha_{j}, \alpha_{d})$ and 
 \begin{eqnarray*}
  \tilde{\Theta}^{(j)}(\mathbf{s}) = \int_{\mathbf{S}^{d}} 
  |\langle \mathbf{s}, \tilde{\mathbf{y}} + y_{d} \tilde{\mathbf{m}}_{d}  \rangle|^{\tilde{\alpha}_{j}} 
  \tilde{\lambda}_{j} ({\rm d} \mathbf{y}),
 \end{eqnarray*}

 \noindent where $\tilde{\lambda}_{j} = \mathds{1}_{\{ \tilde{\alpha}_{j} = \alpha_{j}  \}} \lambda_{j}  +  \mathds{1}_{\{ \tilde{\alpha}_{j} = \alpha_{d}  \}}  \frac{m_{jd}}{1-m_{dd}} \lambda_{d}$,
 $\mathbf{y} = (y_{1}, \dots, y_{d}) \in \mathbb{R}^{d}$ and 
 $\tilde{\mathbf{y}} = (y_{1}, \dots, y_{d-1}) \in \mathbb{R}^{d-1}$. 
 \end{lemma}
 
 It is important to stress that $\tilde{\lambda}_{j} \equiv 0$ when $j,d \in [d] \setminus \Delta$, 
 and otherwise it is non-zero (recall the last comment after the introduction of the main 
 assumptions in Section \ref{ass}).  

\begin{proof}
The fact that $\tilde{\Pi}(F)$ is a non-degenerate, irreducible, critical $(d-1)$-type GW  
forest follows from Lemma 3 (i) in \cite{Gr}. Moreover, we deduce from 
this same lemma (see specifically equations (8) and (9) in \cite{Gr})
that the offspring distribution $\tilde{\bm{\mu}} = (\tilde{\mu}^{(1)}, \dots, \tilde{\mu}^{(d-1)})$ 
has Laplace exponents
\begin{eqnarray*}
 \tilde{\psi}^{(j)}(\mathbf{s}) = \psi^{(j)}(\mathbf{s}, \tilde{\psi}^{(d)}(\mathbf{s}) ),
\end{eqnarray*}

\noindent for $j \in [d-1]$ and $\mathbf{s} \in \mathbb{R}_{+}^{d-1}$, where 
$\tilde{\psi}^{(d)}$ is implicitly defined by
\begin{eqnarray*}
 \tilde{\psi}^{(d)}(\mathbf{s}) = \psi^{(d)}(\mathbf{s}, \tilde{\psi}^{(d)}(\mathbf{s})). 
\end{eqnarray*}

\noindent This is obtained by separating the offspring of each individual with 
types equal and different from $d$. 

In order to understand the behavior of $\tilde{\psi}^{(j)}$ close
to zero, we start by analyzing the one of $\tilde{\psi}^{(d)}$. 
In this direction, we observe from our main assumptions on 
the offspring distribution ${\bm \mu}$ that
\begin{eqnarray*}
 \tilde{\psi}^{(d)}(\mathbf{s}) & = & (1-m_{dd})\langle \tilde{\mathbf{m}}_{d}, \mathbf{s} \rangle + m_{dd}\tilde{\psi}^{(d)}(\mathbf{s}) +
  |(\mathbf{s}, \tilde{\psi}^{(d)}(\mathbf{s}))|^{\alpha_{d}} \Theta^{(d)} \left( \frac{(\mathbf{s}, \tilde{\psi}^{(d)}(\mathbf{s}))}{ |(\mathbf{s}, \tilde{\psi}^{(d)}(\mathbf{s}))|} \right)+o(|(\mathbf{s}, \tilde{\psi}^{(d)}(\mathbf{s}))|^{\alpha_{d}}) \\
  & = & \langle \tilde{\mathbf{m}}_{d}, \mathbf{s} \rangle + \frac{1}{1-m_{dd}}|(\mathbf{s}, \tilde{\psi}^{(d)}(\mathbf{s}))|^{\alpha_{d}} \Theta^{(d)} \left( \frac{(\mathbf{s}, \tilde{\psi}^{(d)}(\mathbf{s}))}{ |(\mathbf{s}, \tilde{\psi}^{(d)}(\mathbf{s}))|} \right)+o(|(\mathbf{s}, \tilde{\psi}^{(d)}(\mathbf{s}))|^{\alpha_{d}}),
  \end{eqnarray*}
  
 \noindent as $|\mathbf{s}| \downarrow 0$.  We also notice that
\begin{eqnarray} \label{eq1}
 \tilde{\psi}^{(d)}(\mathbf{s}) = \langle \tilde{\mathbf{m}}_{d}, \mathbf{s} \rangle + o(|\mathbf{s}| ), 
 \hspace*{6mm} \text{as} \, \, |\mathbf{s}| \downarrow 0.
\end{eqnarray}

On the one hand, from the above estimate, we know that
\begin{eqnarray*}
 \langle (\mathbf{s}, \tilde{\psi}^{(d)}(\mathbf{s})), \mathbf{y} \rangle = \left\langle \mathbf{s}, \tilde{\mathbf{y}} + y_{d} \tilde{\mathbf{m}}_{d} \right \rangle + y_{d} o(  |\mathbf{s}|  ),
\hspace*{6mm} \text{as} \, \, |\mathbf{s}| \downarrow 0,
 \end{eqnarray*}

\noindent Thus, 
 \begin{eqnarray*}
 |(\mathbf{s}, \tilde{\psi}^{(d)}(\mathbf{s}))|^{\alpha_{d}} \Theta^{(d)} \left( \frac{(\mathbf{s}, \tilde{\psi}^{(d)}(\mathbf{s}))}{ |(\mathbf{s}, \tilde{\psi}^{(d)}(\mathbf{s}))|} \right) & = &\int_{\mathbf{S}^{d}} |\langle (\mathbf{s}, \tilde{\psi}^{(d)}(\mathbf{s})), \mathbf{y} \rangle |^{\alpha_{d}} \lambda_{d}({\rm d} \mathbf{y}) \\
 & = & \int_{\mathbf{S}^{d}} \left| \left\langle \mathbf{s}, \tilde{\mathbf{y}} + y_{d} \tilde{\mathbf{m}}_{d} \right \rangle  \right|^{\alpha_{d}} \lambda_{d}({\rm d} \mathbf{y}) +o(|\mathbf{s}|^{\alpha_{d}}) 
 \end{eqnarray*}

On the other hand, from (\ref{eq1}), we have that
\begin{eqnarray*}
 \langle (\mathbf{s}, \tilde{\psi}^{(d)}(\mathbf{s})), (\mathbf{s}, \tilde{\psi}^{(d)}(\mathbf{s})) \rangle = \left\langle \mathbf{s}, \mathbf{s} \right \rangle  +  \langle \mathbf{s}, \tilde{\mathbf{m}}_{d} \rangle^{2} + o(  |\mathbf{s}|^{2}  ),
\hspace*{6mm} \text{as} \, \, |\mathbf{s}| \downarrow 0.
 \end{eqnarray*}

\noindent Then, the previous estimates yields to
\begin{eqnarray} \label{eq2}
 \tilde{\psi}^{(d)}(\mathbf{s}) = \langle \tilde{\mathbf{m}}_{d}, \mathbf{s} \rangle  +
 \frac{1}{1-m_{dd}} |\mathbf{s}|^{\alpha_{d}} \tilde{\Theta}^{(d)} \left( \mathbf{s}/ |\mathbf{s}| \right)+o(|\mathbf{s}|^{\alpha_{d}}),
\hspace*{6mm} |\mathbf{s}| \downarrow 0,
  \end{eqnarray}
  
\noindent where 
\begin{eqnarray*}
 \tilde{\Theta}^{(d)}(\mathbf{s}) = \int_{\mathbf{S}^{d}} \left| \left\langle \mathbf{s}, \tilde{\mathbf{y}} + y_{d} \tilde{\mathbf{m}}_{d} \right \rangle  \right|^{\alpha_{d}} \lambda_{d}({\rm d} \mathbf{y}),  
 \hspace*{6mm} \text{for} \, \,  \mathbf{s} \in \mathbb{R}_{+}^{d-1}.
\end{eqnarray*}

 \noindent Finally, from (\ref{eq1}), (\ref{eq2}) and our assumption on the 
 Laplace exponent $\psi^{(j)}$, the claim follows by similar computations.
\end{proof}

We notice that after performing the $d$- to $(d-1)$-type operation, we are 
left with a non-degenerate, irreducible, critical 
$(d-1)$-type GW forest whose offspring distribution $\tilde{{\bm \mu}}$ has mean matrix 
$\tilde{\mathbf{M}} = (\tilde{m}_{jk})_{j,k \in [d-1]}$. Lemma \ref{lemma1} shows that 
this matrix has spectral radius $1$ and moreover, it is not difficult to check that
its left and right $1$-eigenvectors
$\tilde{\mathbf{a}}$, $\tilde{\mathbf{b}}$ satisfying $\langle  \tilde{\mathbf{a}}, \mathbf{1} \rangle =
\langle \tilde{\mathbf{a}}, \tilde{\mathbf{b}} \rangle= 1$ are given by
\begin{eqnarray*}
 \tilde{\mathbf{a}} = \frac{1}{1-a_{d}} (a_{1}, \dots, a_{d}) \hspace*{5mm} \text{and}
 \hspace*{5mm} \tilde{\mathbf{b}} = \frac{1-a_{d}}{1-a_{d}b_{d}} (b_{1}, \dots, b_{d}). 
\end{eqnarray*}

We are now able to establish Proposition \ref{pro1}.

\begin{proof}[Proof of Proposition \ref{pro1}] The fact that $\Pi^{(i)}(F)$ is a monotype 
GW forest with critical non-degenerate offspring distribution is a consequence of 
Lemma \ref{lemma1} by following exactly the same argument as the proof 
of Proposition 4 (i) in \cite{Gr}. Roughly speaking, the idea is to remove the types
different from $i$ one by one through the $d$- to $(d-1)$-type operation, 
and noticing that the hypotheses of the GW forest under consideration are conserved 
at every step until we are left with a critical non-degenerate monotype GW forest. 
This immediately
shows by induction that the offspring distribution of $\Pi^{(i)}(F)$ is in the domain of attraction 
of a stable law of index $\underline{\alpha} = \min_{j \in [d]} \alpha_{j}$. Thus,
what only remains to be proved is the expression for the Laplace exponent of the offspring 
distribution. 

To this end, recall the notation of Proposition \ref{pro1}. Let 
\begin{eqnarray*}
 \tilde{{\bm \Theta}} (\mathbf{s}) = \left( \tilde{\Theta}^{(1)}(\mathbf{s}) \mathds{1}_{\{ \underline{\alpha} = \tilde{\alpha}_{1} \}}, \dots
 \tilde{\Theta}^{(d-1)} (\mathbf{s}) \mathds{1}_{\{ \underline{\alpha} = \tilde{\alpha}_{d-1} \}} \right),
\end{eqnarray*}

\noindent where $\mathbf{s} \in \mathbb{R}_{+}^{d-1}$. We first observe that 
for $j \in [d-1]$, we have
\begin{eqnarray*}
 \tilde{\Theta}^{(j)}(\tilde{\mathbf{b}}) & = &  \int_{\mathbf{S}^{d}}  |\langle \tilde{\mathbf{b}}, \tilde{\mathbf{y}} + y_{d} \tilde{\mathbf{m}}_{d}  \rangle|^{\tilde{\alpha}_{j}} \tilde{\lambda}_{j} ({\rm d} \mathbf{y}) \\
 & = & \int_{\mathbf{S}^{d}} \left| \langle \tilde{\mathbf{b}}, \tilde{\mathbf{y}} \rangle + y_{d} \langle  \tilde{\mathbf{b}}, \tilde{\mathbf{m}}_{d} \rangle \right|^{\tilde{\alpha}_{j}} \tilde{\lambda}_{j} ({\rm d} \mathbf{y}) \\
 & = & \left(\frac{1-a_{d}}{1-a_{d}b_{d}} \right)^{\tilde{\alpha}_{j}} \int_{\mathbf{S}^{d}} \left| \sum_{k=1}^{d-1} b_{k} y_{k} + y_{d}\sum_{k=1}^{d-1} b_{k} \frac{m_{dk}}{1-m_{dd}} \right|^{\tilde{\alpha}_{j}} \tilde{\lambda}_{j} ({\rm d} \mathbf{y}) \\
 & = & \left(\frac{1-a_{d}}{1-a_{d}b_{d}} \right)^{\tilde{\alpha}_{j}} \tilde{\Theta}^{(j)}(\mathbf{b}),
\end{eqnarray*}

\noindent where for the last equality, we use the fact the $\mathbf{b}$ is the 
right $1$-eigenvector of the mean matrix $\mathbf{M}$, that is, $\sum_{k \in [d]} b_{k} m_{dk} = b_{d}$.
Then, from the previous identity, we have that
\begin{eqnarray*}
 \langle \tilde{\mathbf{a}}, \tilde{{\bm \Theta}} (\tilde{\mathbf{b}}) \rangle 
 & = & \left(\frac{1-a_{d}}{1-a_{d}b_{d}}\right)^{\underline{\alpha}} 
 \left( \sum_{k=1}^{d-1} \tilde{a}_{k} \Theta^{(k)}(\mathbf{b}) \mathds{1}_{\{ \underline{\alpha} = \alpha_{k} \}}
 + \Theta^{(d)}(\mathbf{b}) \mathds{1}_{\{ \underline{\alpha} = \alpha_{d} \}} \sum_{k=1}^{d-1} \tilde{a}_{k} \frac{m_{kd}}{1-m_{dd}} \right) \\
 & = & \frac{(1-a_{d})^{\underline{\alpha}-1}}{(1-a_{d}b_{d})^{\underline{\alpha}}} \langle \mathbf{a}, {\bm \Theta} (\mathbf{b}) \rangle,
\end{eqnarray*}

\noindent where in the last equality, we now use that $\mathbf{a}$ is the 
left $1$-eigenvector of the mean matrix $\mathbf{M}$, i.e., 
$\sum_{k \in [d]} a_{k} m_{kd} = a_{d}$. Therefore, the expression for the Laplace
exponent readily follows by induction on the number of types, making use of
Lemma \ref{lemma1} and the above identity. 
\end{proof}

Following Miermont \cite{Gr}, we are interested in keeping 
the information of the number vertices that we delete 
during the projection $\Pi^{(i)}$. More precisely, for $\mathbf{f} \in \mathbb{F}^{(d)}$, 
recall that $\Pi^{(i)}(\mathbf{f})$ is the monotype forest obtained 
by removing all the vertices with type different from $i$. Then, for a vertex 
$u \in \Pi^{(i)}(\mathbf{f})$ with children $u1, \dots, uk$, we let 
$\mathbf{f}_{v_{u}}, \mathbf{f}_{v_{u1}}, \dots, \mathbf{f}_{v_{uk}}$ be the subtrees
of the original forest $\mathbf{f}$ rooted at $u, u1, \dots, uk$, respectively. Then, we let 
\begin{eqnarray*}
 N_{ij}(u) = \# \left \{ w \in \mathbf{f}_{v_{u}} \setminus \left(\bigcup_{r=1}^{k} \mathbf{f}_{v_{ur}} \right) : e_{\mathbf{f}}(w) = j \right\},
\hspace{6mm} \text{for} \, \, j \in [d] \setminus \{i\},
\end{eqnarray*}

\noindent be the number of type $j$ vertices that have been deleted between $u$ and its
children.  We also let
\begin{eqnarray*}
 \hat{N}_{ij}(n) = \# \left\{v \in \mathbf{f}_{n}: e_{\mathbf{f}}(v) = j \, \, \, \text{and} \, \, \, e_{\mathbf{f}}(w) \neq i \, \, \text{for all} \, \, w \vdash v \right\},
 \hspace{6mm} \text{for} \, \, j \in [d] \setminus \{i\}, 
\end{eqnarray*}

\noindent be the number of type $j$ vertices of the $n$-th tree component of $\mathbf{f}$ that lie below the
first layer of type  $i$ vertices, i.e. the number of type $j$ vertices of $\mathbf{f}_{n}$ that do
not have ancestors of type $i$. 

\definecolor{sqsqsq}{rgb}{0.12549019607843137,0.12549019607843137,0.12549019607843137}
\definecolor{qqwuqq}{rgb}{0.,0.39215686274509803,0.}
\definecolor{qqqqff}{rgb}{0.,0.,1.}
\definecolor{ffqqqq}{rgb}{1.,0.,0.}

\begin{figure}[H]\centering 
\begin{minipage}[b]{10\linewidth}
\begin{tikzpicture}[line cap=round,line join=round,>=triangle 45,x=0.8cm,y=0.8cm]
\clip(-2.5664583333333333,-0.5770833333333383) rectangle (11.339791666666667,6.5322916666666995);
\draw (0.8,3.)-- (1.48,2.);
\draw (1.46,3.)-- (1.48,2.);
\draw (2.2,2.98)-- (1.48,2.);
\draw (1.48,2.)-- (3.,1.);
\draw (2.5,2.)-- (3.,1.);
\draw (3.5,2.)-- (3.,1.);
\draw (3.,1.)-- (4.52,1.98);
\draw (4.5,2.98)-- (4.52,1.98);
\draw (3.82,3.)-- (4.52,1.98);
\draw (5.26,3.)-- (4.52,1.98);
\draw (7.,2.)-- (8.,1.);
\draw (8.,2.)-- (8.,1.);
\draw (9.,2.)-- (8.,1.);
\draw (8.26,2.98)-- (9.,2.);
\draw (9.,3.)-- (9.,2.);
\draw (9.8,3.02)-- (9.,2.);
\draw [color=sqsqsq] (8.22,4.02)-- (8.26,2.98);
\draw [color=sqsqsq] (8.7,4.02)-- (9.,3.);
\draw [color=sqsqsq] (9.32,4.)-- (9.,3.);
\draw [color=sqsqsq] (3.76,4.)-- (3.82,3.);
\draw [color=sqsqsq] (2.7,4.02)-- (2.2,2.98);
\draw [color=sqsqsq] (2.,4.)-- (2.2,2.98);
\draw [color=sqsqsq] (5.28,3.98)-- (5.26,3.);
\draw [color=sqsqsq] (0.78,3.98)-- (0.8,3.);
\draw (8.324166666666667,1.0635416666666715) node[anchor=north west] {$\mathbf{f}_{2}$};
\draw (3.6366666666666667,1.2979166666666726) node[anchor=north west] {$\mathbf{f}_{1}$};
\draw (-1.5039583333333333,5.438541666666695) node[anchor=north west] {$\mathbf{f}:$};
\draw [color=sqsqsq] (4.48752501031551,4.025564518714591)-- (4.5,2.98);
\draw (-2.3789583333333333,2.7822916666666804) node[anchor=north west] {\parbox{2.1625 cm}{\footnotesize{$N_{12}(u(3)) = 4 \\ N_{13}(u(3)) = 1$}}};
\draw (1.5116666666666667,0.422916666666668) node[anchor=north west] {\footnotesize{$\hat{N}_{12}(1) = \hat{N}_{13}(1) =0$}};
\draw (9.042916666666667,1.6260416666666744) node[anchor=north west] {\parbox{2.3125 cm}{\footnotesize{$\hat{N}_{12}(2) = 2 \\ \hat{N}_{13}(2) =0$}}};
\draw [rotate around={-121.94445734456568:(1.955056668799645,3.061524592635171)},line width=1.2pt] (1.955056668799645,3.061524592635171) ellipse (1.6755809554064574cm and 0.6242281294689868cm);
\draw [rotate around={120.45522368931277:(7.541912434733292,1.4050132330918619)},line width=1.2pt] (7.541912434733292,1.4050132330918619) ellipse (0.9223040465633588cm and 0.3740044787167834cm);
\begin{scriptsize}
\draw [fill=ffqqqq] (3.,1.) circle (2.5pt);
\draw[color=ffqqqq] (2.7304166666666667,0.9385416666666707) node {1};
\draw [fill=ffqqqq] (2.5,2.) circle (2.5pt);
\draw[color=ffqqqq] (2.2772916666666667,1.907291666666676) node {9};
\draw [fill=qqqqff,shift={(1.48,2.)}] (0,0) ++(0 pt,3.75pt) -- ++(3.2475952641916446pt,-5.625pt)--++(-6.495190528383289pt,0 pt) -- ++(3.2475952641916446pt,5.625pt);
\draw[color=qqqqff] (1.1991666666666667,1.8760416666666757) node {2};
\draw [fill=qqwuqq] (4.52,1.98) ++(-2.5pt,0 pt) -- ++(2.5pt,2.5pt)--++(2.5pt,-2.5pt)--++(-2.5pt,-2.5pt)--++(-2.5pt,2.5pt);
\draw[color=qqwuqq] (4.136666666666667,2.0479166666666764) node {11};
\draw [fill=ffqqqq] (3.5,2.) circle (2.5pt);
\draw[color=ffqqqq] (3.1679166666666667,1.907291666666676) node {10};
\draw [fill=ffqqqq] (5.26,3.) circle (2.5pt);
\draw[color=ffqqqq] (5.496041666666667,3.2979166666666835) node {16};
\draw [fill=qqwuqq] (4.5,2.98) ++(-2.5pt,0 pt) -- ++(2.5pt,2.5pt)--++(2.5pt,-2.5pt)--++(-2.5pt,-2.5pt)--++(-2.5pt,2.5pt);
\draw[color=qqwuqq] (4.714791666666667,3.2666666666666835) node {14};
\draw [fill=qqqqff,shift={(3.82,3.)}] (0,0) ++(0 pt,3.75pt) -- ++(3.2475952641916446pt,-5.625pt)--++(-6.495190528383289pt,0 pt) -- ++(3.2475952641916446pt,5.625pt);
\draw[color=qqqqff] (3.9804166666666667,3.2822916666666835) node {12};
\draw [fill=qqwuqq] (2.2,2.98) ++(-2.5pt,0 pt) -- ++(2.5pt,2.5pt)--++(2.5pt,-2.5pt)--++(-2.5pt,-2.5pt)--++(-2.5pt,2.5pt);
\draw[color=qqwuqq] (1.9179166666666667,2.9385416666666817) node {6};
\draw [fill=ffqqqq] (0.8,3.) circle (2.5pt);
\draw[color=ffqqqq] (0.5272916666666667,2.907291666666681) node {3};
\draw [fill=qqqqff,shift={(1.46,3.)}] (0,0) ++(0 pt,3.75pt) -- ++(3.2475952641916446pt,-5.625pt)--++(-6.495190528383289pt,0 pt) -- ++(3.2475952641916446pt,5.625pt);
\draw[color=qqqqff] (1.2460416666666667,2.891666666666681) node {5};
\draw [fill=qqqqff,shift={(8.,1.)}] (0,0) ++(0 pt,3.75pt) -- ++(3.2475952641916446pt,-5.625pt)--++(-6.495190528383289pt,0 pt) -- ++(3.2475952641916446pt,5.625pt);
\draw[color=qqqqff] (7.699166666666667,0.9854166666666709) node {18};
\draw [fill=ffqqqq] (9.,2.) circle (2.5pt);
\draw[color=ffqqqq] (9.339791666666667,2.1104166666666773) node {21};
\draw [fill=ffqqqq] (8.,2.) circle (2.5pt);
\draw[color=ffqqqq] (8.261666666666667,2.1572916666666773) node {20};
\draw [fill=qqqqff,shift={(7.,2.)}] (0,0) ++(0 pt,3.75pt) -- ++(3.2475952641916446pt,-5.625pt)--++(-6.495190528383289pt,0 pt) -- ++(3.2475952641916446pt,5.625pt);
\draw[color=qqqqff] (7.230416666666667,2.1416666666666773) node {19};
\draw [fill=qqwuqq] (8.26,2.98) ++(-2.5pt,0 pt) -- ++(2.5pt,2.5pt)--++(2.5pt,-2.5pt)--++(-2.5pt,-2.5pt)--++(-2.5pt,2.5pt);
\draw[color=qqwuqq] (8.464791666666667,3.2354166666666826) node {22};
\draw [fill=qqwuqq] (9.,3.) ++(-2.5pt,0 pt) -- ++(2.5pt,2.5pt)--++(2.5pt,-2.5pt)--++(-2.5pt,-2.5pt)--++(-2.5pt,2.5pt);
\draw[color=qqwuqq] (9.277291666666667,3.2354166666666826) node {24};
\draw [fill=qqqqff,shift={(9.8,3.02)}] (0,0) ++(0 pt,3.75pt) -- ++(3.2475952641916446pt,-5.625pt)--++(-6.495190528383289pt,0 pt) -- ++(3.2475952641916446pt,5.625pt);
\draw[color=qqqqff] (9.996041666666667,3.2510416666666835) node {S};
\draw [fill=ffqqqq] (5.28,3.98) circle (2.5pt);
\draw[color=ffqqqq] (5.589791666666667,3.954166666666687) node {17};
\draw [fill=qqqqff,shift={(2.,4.)}] (0,0) ++(0 pt,3.75pt) -- ++(3.2475952641916446pt,-5.625pt)--++(-6.495190528383289pt,0 pt) -- ++(3.2475952641916446pt,5.625pt);
\draw[color=qqqqff] (2.2460416666666667,3.860416666666686) node {7};
\draw [fill=qqqqff,shift={(2.7,4.02)}] (0,0) ++(0 pt,3.75pt) -- ++(3.2475952641916446pt,-5.625pt)--++(-6.495190528383289pt,0 pt) -- ++(3.2475952641916446pt,5.625pt);
\draw[color=qqqqff] (2.9179166666666667,3.860416666666686) node {8};
\draw [fill=qqwuqq] (3.76,4.) ++(-2.5pt,0 pt) -- ++(2.5pt,2.5pt)--++(2.5pt,-2.5pt)--++(-2.5pt,-2.5pt)--++(-2.5pt,2.5pt);
\draw[color=qqwuqq] (3.9960416666666667,3.922916666666687) node {13};
\draw [fill=ffqqqq] (8.22,4.02) circle (2.5pt);
\draw[color=ffqqqq] (8.386666666666667,4.313541666666689) node {23};
\draw [fill=ffqqqq] (8.7,4.02) circle (2.5pt);
\draw[color=ffqqqq] (8.886666666666667,4.297916666666689) node {25};
\draw [fill=qqwuqq] (9.32,4.) ++(-2.5pt,0 pt) -- ++(2.5pt,2.5pt)--++(2.5pt,-2.5pt)--++(-2.5pt,-2.5pt)--++(-2.5pt,2.5pt);
\draw[color=qqwuqq] (9.511666666666667,4.282291666666689) node {26};
\draw [fill=ffqqqq] (0.78,3.98) circle (2.5pt);
\draw[color=ffqqqq] (0.9960416666666667,3.860416666666686) node {4};
\draw [fill=qqwuqq] (4.48752501031551,4.025564518714591) ++(-2.5pt,0 pt) -- ++(2.5pt,2.5pt)--++(2.5pt,-2.5pt)--++(-2.5pt,-2.5pt)--++(-2.5pt,2.5pt);
\draw[color=qqwuqq] (4.714791666666667,3.922916666666687) node {15};
\draw [fill=black] (6.841710221003097,2.204045190814083) ++(-0.5pt,0 pt) -- ++(0.5pt,0.5pt)--++(0.5pt,-0.5pt)--++(-0.5pt,-0.5pt)--++(-0.5pt,0.5pt);
\draw [fill=black] (1.3702051765544383,3.5948467523918044) ++(-0.5pt,0 pt) -- ++(0.5pt,0.5pt)--++(0.5pt,-0.5pt)--++(-0.5pt,-0.5pt)--++(-0.5pt,0.5pt);
\draw [fill=black] (2.577604334407634,2.5861335319068557) ++(-0.5pt,0 pt) -- ++(0.5pt,0.5pt)--++(0.5pt,-0.5pt)--++(-0.5pt,-0.5pt)--++(-0.5pt,0.5pt);
\draw [fill=black] (3.21951274744351,4.603559972876753) ++(-0.5pt,0 pt) -- ++(0.5pt,0.5pt)--++(0.5pt,-0.5pt)--++(-0.5pt,-0.5pt)--++(-0.5pt,0.5pt);
\draw [fill=black] (0.6671626289437169,1.8372403833649997) ++(-0.5pt,0 pt) -- ++(0.5pt,0.5pt)--++(0.5pt,-0.5pt)--++(-0.5pt,-0.5pt)--++(-0.5pt,0.5pt);
\draw [fill=black] (1.217369840117325,1.27174963854768) ++(-0.5pt,0 pt) -- ++(0.5pt,0.5pt)--++(0.5pt,-0.5pt)--++(-0.5pt,-0.5pt)--++(-0.5pt,0.5pt);
\draw [fill=black] (8.24779531622454,0.6451247591555131) ++(-0.5pt,0 pt) -- ++(0.5pt,0.5pt)--++(0.5pt,-0.5pt)--++(-0.5pt,-0.5pt)--++(-0.5pt,0.5pt);
\draw [fill=black] (7.177947961164746,1.1036307684668536) ++(-0.5pt,0 pt) -- ++(0.5pt,0.5pt)--++(0.5pt,-0.5pt)--++(-0.5pt,-0.5pt)--++(-0.5pt,0.5pt);
\draw [fill=black] (7.636453970476086,2.0664933880206684) ++(-0.5pt,0 pt) -- ++(0.5pt,0.5pt)--++(0.5pt,-0.5pt)--++(-0.5pt,-0.5pt)--++(-0.5pt,0.5pt);
\draw [fill=black] (7.880990508775467,1.74553918150273) ++(-0.5pt,0 pt) -- ++(0.5pt,0.5pt)--++(0.5pt,-0.5pt)--++(-0.5pt,-0.5pt)--++(-0.5pt,0.5pt);
\end{scriptsize}
\end{tikzpicture}
\end{minipage}
\caption{{\small A representation of the quantities $N_{1j}$ and $\hat{N}_{2j}$,  
for a three-type planar forest with two tree components, type $1$ vertices 
represented with circles, type $2$ vertices with triangles and type $3$ vertices 
with diamonds.}} \label{fig2}
\end{figure}
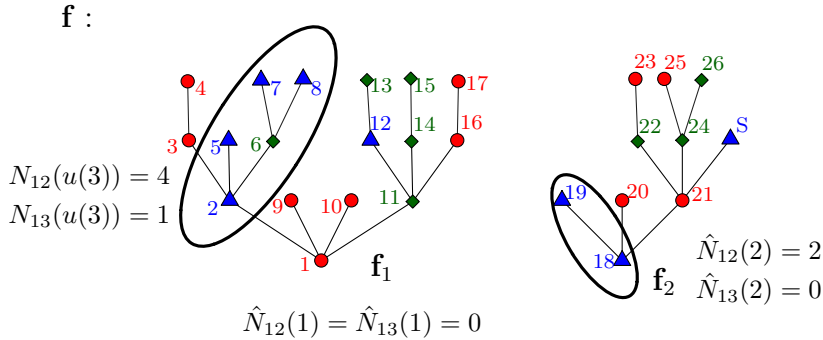

The following proposition provides information about the distribution of the previous
quantities.  

\begin{proposition} \label{pro2}
Let $1=u(0) \prec  u(1) \prec \dots \prec u(\# \Pi^{(i)}(\mathbf{f})-1)$
be the list of vertices of $\Pi^{(i)}(\mathbf{f})$ in depth-first order and let 
$ \mathbf{x} \in [d]^{\mathbb{N}}$. Then, under the law $ \mathbf{P}^{\mathbf{x}}$ and 
for each $ i \in [d]$:

\begin{itemize}
  \item[(i)] For every $j \in [d] \setminus \{i\}$, the random variables $(N_{ij}(u(n)), n \geq 0)$ are i.i.d. 
  Moreover, their Laplace exponents satisfy
\begin{eqnarray*}
 \phi_{ij}(s) := - \log \mathbf{E}^{\mathbf{x}} \left[\exp \left(-s N_{ij}(u(0)) \right) \right] = \frac{a_{j}}{a_{i}} s + 
 c_{ij}s^{\underline{\alpha}} + o(s^{\underline{\alpha}}), \hspace*{6mm} \text{as} \, \, s \downarrow 0,
\end{eqnarray*}

\noindent where $s \in \mathbb{R}_{+}$, $\underline{\alpha} = \min_{j \in [d]} \alpha_{j}$ and
$c_{ij} > 0$ a constant. In particular, 
$\mathbf{E}^{\mathbf{x}}[N_{ij}(u(0))] = a_{j}/a_{i}$. 

 \item[(ii)] For every $j \in [d] \setminus \{i\}$, the random variables $(\hat{N}_{ij}(n), n \geq 1)$ 
 are independent, and their Laplace exponents satisfy
 \begin{eqnarray*}
 \hat{\phi}_{ij}(s) := - \log \mathbf{E}^{\mathbf{x}} \left[\exp \left(-s \hat{N}_{ij}(n) \right) \right] = \left( \hat{c}_{ij} s + 
 \hat{c}_{ij}^{\prime}s^{\hat{\alpha}_{i}} + o(s^{\hat{\alpha}_{i}}) \right) \mathds{1}_{\{x_{n} \neq i \}}, \hspace*{6mm} \text{as} \, \, s \downarrow 0,
\end{eqnarray*}

\noindent for $s \in \mathbb{R}_{+}$, some constants $\hat{c}_{ij} > 0$ and $\hat{c}_{ij}^{\prime} \geq 0$ (that depends of $x_{n}$) 
and where $\hat{\alpha}_{i} = \min_{j \in [d] \setminus \{i\}} \alpha_{j}$.
\end{itemize}
\end{proposition}

\begin{proof}
 \item[(i)] The fact that for every $j \in [d] \setminus \{i\}$, the random variables $(N_{ij}(u(n)), n \geq 0)$ are i.i.d.
  has been proven in Proposition 4 (ii) of \cite{Gr}. Basically, this follows
  from Jagers' theorem on stopping lines \cite{Ja}. We then focus on the second part of the statement, and for simplicity, we 
prove this in the case $i=1$, without losing generality. The idea is based in a 
similar induction argument as in the proof of Proposition \ref{pro1}, 
by making use of the $d$- to $(d-1)$-type operation $\tilde{\Pi}$. In this direction, 
for $\mathbf{f} \in \mathbb{F}^{(d)}$ and $u \in \tilde{\Pi}(\mathbf{f})$,
we let $\tilde{N}(u)$ be the number of $d$-type vertices that have been deleted between
$u$ and its children during this procedure. For $j \in [d-1]$,
we let $u^{(j)}(0) \prec u^{(j)}(1) \prec \dots$ be the type $j$ vertices of
$F$ arranged in depth-first order. Then, Lemma 3 (ii) in \cite{Gr} ensures that under
$\mathbf{P}^{\mathbf{x}}$, the $d-1$ sequences $(\tilde{N}(u^{(j)}(n)), n \geq 0)$
are independent and formed of i.i.d. elements. Further, their Laplace exponents
$\tilde{\phi}^{(j)}$ respectively satisfy
\begin{eqnarray*}
 \tilde{\phi}^{(j)}(s) = \psi^{(j)}(\mathbf{0},\tilde{\phi}^{(d)}(s))
\end{eqnarray*}

\noindent for $s \in \mathbb{R}_{+}$, $\mathbf{0}$ the vector
of $\mathbb{R}_{+}^{d-1}$ with all components equal to $0$, and where 
$\tilde{\phi}^{(d)}$ is implicitly given by 
\begin{eqnarray} \label{eq3}
 \tilde{\phi}^{(d)}(s) = s + \psi^{(d)}(\mathbf{0},\tilde{\phi}^{(d)}(s)).
\end{eqnarray}
 
\noindent Thus, from our main assumptions on the offspring distribution, it 
is not difficult to check by following the same reasoning as the proof
of Lemma \ref{lemma1} that
\begin{eqnarray*}
 \tilde{\phi}^{(j)}(s) = \frac{m_{jd}}{1-m_{dd}} s + 
 \tilde{c}_{jd} s^{\tilde{\alpha}_{j}} + o(s^{\tilde{\alpha}_{j}}), \hspace*{6mm} \text{as} \, \, s \downarrow 0,
\end{eqnarray*}

\noindent where $\tilde{\alpha}_{j} = \min(\alpha_{j}, \alpha_{d})$
and the constant $\tilde{c}_{jd} = 0$ if $j,d \in [d] \setminus \Delta$ and $\tilde{c}_{jd} > 0$ 
otherwise (recall the main assumptions ($\mathbf{H}_{2}.\mathbf{1}$) and ($\mathbf{H}_{2}.\mathbf{2}$)). 

Let now proceed to prove our statement. In the monotype case, $d =1$, 
there is nothing to show. For the case $d = 2$, one checks 
from the previous discussion that the Laplace exponent of $N_{12}(u(0))$ satisfies
\begin{eqnarray*}
 \phi_{12}(s) = \frac{m_{12}}{1-m_{22}} s + 
 \tilde{c}_{12}s^{\tilde{\alpha}_{1}} + o(s^{\tilde{\alpha}_{1}}), \hspace*{6mm} \text{as} \, \, s \downarrow 0.
\end{eqnarray*}

\noindent On the other hand, a simple computation shows that $m_{12}/(1-m_{22}) = 
a_{2}/a_{1}$. 

We now consider case $d \geq 3$. We apply the operation $\tilde{\Pi}$, 
$d-2$ times, removing the types $d, d-1, \dots, 3$ one after the other.
We then obtain a two-type GW forest and we observe that the 
number of type $2$ vertices that have only the root as type $1$ ancestor is precisely the number of 
type $2$ individuals that are trapped between two generations of $\Pi^{(1)}(F)$. Therefore, 
in view of the $d=2$ case above, it is not difficult to see that the Laplace exponent of 
$N_{12}(u(0))$ satisfies
\begin{eqnarray*}
 \phi_{12}(s) = \frac{a_{2}}{a_{1}} s + 
 c_{12}s^{\underline{\alpha}} + o(s^{\underline{\alpha}}), \hspace*{6mm} \text{as} \, \, s \downarrow 0,
\end{eqnarray*}

\noindent for some constant $c_{12} > 0$. Finally, our claim follows by symmetry.

\item[(ii)] This is obtained by a similar induction argument. We only need
to notice that for $i \in [d]$ and $j \in [d] \setminus \{i\}$, 
$\hat{N}_{ij}(n) = 0$ when $x_{n} = i$.

\end{proof}

\subsection{Sub-exponential Bounds} \label{subs2}

The following lemma gives an exponential control on 
the height and number of components related to the $n$ first vertices in $d$-type
GW forests. This extends Lemma 4 in \cite{Gr} which considers 
the finite variance case. Recall that for a forest $\mathbf{f} \in \mathbb{F}$, 
we let $1 \prec u_{\mathbf{f}}(0) \prec u_{\mathbf{f}}(1) \prec \cdots \prec
u_{\mathbf{f}}(\# \mathbf{f}-1)$ be the depth-first ordered list of its vertices. Recall
also that $\Upsilon_{n}^{\mathbf{f}}$ is the index of the tree component to with 
$u_{\mathbf{f}}(n)$ belongs.

\begin{lemma} \label{lemma2}
  There exist two constants $0 < C_{1}, C_{2} < \infty$ (depending only on 
 ${\bm \mu}$) such that for every $n \in \mathbb{N}$, $\mathbf{x} \in [d]^{\mathbb{N}}$ and
 $\eta > 0$, 
 \begin{eqnarray*} 
  \mathbf{P}^{\mathbf{x}}\left( \max_{0\leq k \leq n } |u_{F}(k)| \geq n^{1 -1/\underline{\alpha} + \eta} \right) \leq C_{1}(n+1) \exp \left(-C_{2}n^{\eta} \right)
 \end{eqnarray*}
 
 \noindent and
 \begin{eqnarray*} 
  \mathbf{P}^{\mathbf{x}}\left( \Upsilon_{n}^{F} \geq n^{1 -1/\underline{\alpha} +\eta}\right) \leq C_{1} \exp \left(-C_{2}n^{\eta} \right).
 \end{eqnarray*}
 \end{lemma}

 \begin{proof}
  We observe that under  $\mathbf{P}^{\mathbf{x}}$ and independently of $\mathbf{x}$, we
  have that 
 \begin{eqnarray*} 
  \max_{0\leq k \leq n } |u_{F}(k)| \leq \sum_{i \in [d]} \max_{0\leq k \leq n } \left |u_{\Pi^{(i)}(F)}(k) \right| \hspace*{5mm} \text{and} \hspace*{5mm} \Upsilon_{n}^{F} \leq \sum_{i \in [d]} \Upsilon_{n}^{\Pi^{(i)}(F)}, 
  \end{eqnarray*}
 
 \noindent where each of the forests $\Pi^{(i)}(F)$, for $i \in [d]$, are critical non-degenerate monotype GW forests with 
 offspring distribution in the domain of attraction of a stable law of index $\underline{\alpha} \in (1,2]$
 by Proposition \ref{pro1}. Therefore, from the above inequalities, it is enough to 
 prove the result only for the case $d=1$. 
 
 In this direction, let $\mu$ be a critical non-degenerate offspring distribution on 
 $\mathbb{Z}_{+}$, with Laplace exponent given by
 \begin{eqnarray*}
  \psi(s) = s + cs^{\alpha}+o(s^{\alpha}), \hspace*{6mm} \text{as} \, \, s \downarrow 0,
 \end{eqnarray*}

 \noindent for $\alpha \in (1,2]$, $s \in \mathbb{R}_{+}$ and $c >0$ a constant. 
 Let $\mathbf{P}$ be the law 
 of a monotype GW forest with an infinite number of components and offspring
 distribution $\mu$. We then let $F$ be a monotype GW forest with law $\mathbf{P}$. 
 
 It is well-known (\cite{Du}, Section 2.2) that $|u_{F}(k)|-1$ has the same distribution
  as the number of weak records for a random walk with step distribution 
  $\mu(\{ \cdot+1 \})$ on $\{-1\} \cup \mathbb{Z}_{+}$, from time $1$ up to time $k$. 
  We denote by $(W_{n}, n\geq 0)$ such random walk and we also consider that   
  is defined on some probability space $(\Omega, \mathcal{A}, \mathbb{P})$. 
  By assumption, the step distribution of this random walk is centered and in the domain of 
  attraction of stable law of index $\alpha \in (1,2]$. That is,  
  $W_{n}/n^{1/\alpha}$ converges in distribution towards a stable law of index $\alpha$ 
  as $n \rightarrow \infty$. We fix $\tau_{0} = 0$ and write 
  $\tau_{j}$, $j \geq 0$, for the time of the $j$-th weak record of 
  $(W_{n}, n \geq 0)$. Therefore, from \cite{Fe} and Theorems 1 and 2 in \cite{Do}, 
  the sequence of random variables $(\tau_{j} - \tau_{j-1}, j \geq 1)$ is 
  i.i.d. with Laplace exponent given by
  \begin{eqnarray} \label{eq7}
   \tilde{\kappa}(\lambda) = - \log \mathbb{E}\left[ \exp \left( -\lambda \tau_{1} \right) \right] = \tilde{C}_{1} \lambda^{1-1/\alpha} + o(\lambda^{1-1/\alpha}),
   \hspace*{6mm} \text{as} \, \, \lambda \downarrow 0,
  \end{eqnarray}

  \noindent for some constant $\tilde{C}_{1} > 0$. 
  We then bound the first probability by 
  \begin{eqnarray*} 
  \mathbf{P}\left( \max_{0\leq k \leq n } |u_{F}(k)| \geq n^{1-1/\alpha+ \eta} \right) \leq (n+1) \max_{ 0 \leq k \leq n } \mathbf{P}\left( |u_{F}(k)| \geq n^{1-1/\alpha+ \eta} \right).  
 \end{eqnarray*}

\noindent Then, we notice that for $0 \leq k \leq n$ and $m \in \mathbb{N}$, we have that
  \begin{eqnarray*}
   \mathbf{P}\left(|u_{F}(k)| - 1 \geq m \right)  =  \mathbb{P}\left( \sum_{j=1}^{m} \left(\tau_{j} - \tau_{j-1} \right) \leq k \right) \leq e \mathbb{E} \left[\exp \left(-\sum_{j=1}^{m} \frac{\tau_{j}-\tau_{j-1}}{k} \right) \right] \leq  \exp\left(1-m \tilde{\kappa}(1/n) \right),
  \end{eqnarray*}
  
  \noindent where for the last inequality, we use the monotonicity of $\tilde{\kappa}$. Taking
  $m = \left \lceil n^{1-1/\alpha+ \eta} \right \rceil-1$ and using (\ref{eq7}), we get the first 
  bound for large $n$ and thus for every $n$ up to tuning the constants $C_{1}, C_{2}$. 
  
 The proof for second bound is very similar. For $j \geq 1$, let $\# F_{j}$
 be the number of vertices of the $j$-th tree component of the forest $F$. By the Otter-Dwass formula (see, 
 e.g., \cite{Pi}, Chapter 5), under $\mathbf{P}$, $(\# F_{i}, i \geq 1)$
 is a sequence of i.i.d. random variables with common distribution
 \begin{eqnarray*}
  \mathbf{P}\left(\# F_{1} = n \right) = n^{-1}  \mathbb{P}\left(W_{n} = -1 \right).
 \end{eqnarray*}

 \noindent Using again the fact that the step distribution of $(W_{n}, n \geq 0)$ is
 centered and in the domain of attraction of a stable law of index $\alpha$, we obtain 
 that
 \begin{eqnarray*}
  \mathbf{P}\left(\# F_{1} = n \right) = \tilde{C}_{2} n^{-1-1/\alpha} + o(n^{-1-1/\alpha}),
  \hspace*{6mm} \text{as} \, \, n \rightarrow \infty,
 \end{eqnarray*}

 \noindent where $\tilde{C}_{2} > 0$ is some positive constant; see for example Lemma 1 in \cite{Ko}. 
 Therefore, an  Abelian theorem (\cite{Fe}, Theorem XIII.5.5) entails that the Laplace 
 exponent $\kappa$ of the distribution of $\# F_{1}$, under $\mathbf{P}$, satisfies
  \begin{eqnarray} 
  \kappa(\lambda)  = \tilde{C}_{3} \lambda^{1-1/\alpha} + o(\lambda^{1-1/\alpha}),
  \hspace*{6mm} \text{as} \, \, \lambda \downarrow 0,
  \end{eqnarray}
 for some constant $\tilde{C}_{3} > 0$. Noticing that $\left\{ \Upsilon_{n}^{F}(n) \geq m \right\} = \left\{ \sum_{i=1}^{m-1} \# F_{i} \leq n \right\}$,
 the second bound is then obtained analogously as the first one. Finally, we
 tune up the constants $C_{1}, C_{2}$ so that they match to both cases. 
 \end{proof}
 
\subsection{Convergence of types} \label{subs3}

In order to compare the height process of the monotype GW forest 
$\Pi^{(i)}(F)$, $i \in [d]$, with that of the $d$-type GW forest $F$, we must 
estimate the number of vertices of $F$ that stand between a type $i$ vertex of 
$\Pi^{(i)}(F)$ and one of its descendants. This is the purpose of the following
result. Before that, we need some further notation.

\begin{definition}
We say that a sequence of positive numbers $(z_{n}, n \geq 0)$
 is exponentially bounded if there are positive constants $c, C >0$ such that
 $z_{n} \leq C e^{-cn^{\epsilon}}$ for some $\varepsilon >0$ and large enough $n$. In order to simplify notations and 
 avoid referring to the changing $\varepsilon$'s and the constants $c$ and $C$, we write 
 $z_{n} = \text{oe}(n)$ in this case.
 \end{definition}

For a $d$-type forest $\mathbf{f} \in \mathbb{F}^{(d)}$ and a vertex $u \in \mathbf{f}$,
we let $\text{Anc}_{\mathbf{f}}^{u}(i)$ be the number of type $i$ ancestors of a vertex $u$. 
Proposition 5 in \cite{Gr} provides the following key estimate for the height process. 

\begin{proposition} \label{pro3}
 For every $\gamma >0$ and $\mathbf{x} \in [d]^{\mathbb{N}}$, we have that
 \begin{eqnarray*} 
  \max_{i \in [d]} \mathbf{P}^{\mathbf{x}} \left( \max_{0 \leq k \leq n} \left| H_{k}^{F} - \frac{{\rm Anc}_{F}^{u(k)}(i)}{a_{i}b_{i}} \right| >  n^{1/2 - 1/2\underline{\alpha} + \gamma} \right) = {\rm oe}(n).
 \end{eqnarray*}
\end{proposition}

On the other hand, observe that the height process of the monotype GW forest 
$\Pi^{(i)}(F)$ does not visit the vertices of type different from $i$, in words, 
it goes faster than the the height process of the $d$-type GW forest $F$. Then, 
in order to slow down the height process of $\Pi^{(i)}(F)$, we must adjust the 
time. We conclude this section with the following result which takes care of the 
number of vertices with type different from $i$ that stands between two consecutive
type $i$ vertices in $\Pi^{(i)}(F)$. More precisely, for  $\mathbf{f} \in \mathbb{F}^{(d)}$ 
and $n \geq 0$, we let
\begin{eqnarray*}
\Lambda_{i}^{\mathbf{f}}(n) = \# \left \{ 0 \leq k \leq n: e_{\mathbf{f}}(u_{\mathbf{f}}(k)) = i\right \}
\end{eqnarray*}

\noindent be the number of type $i$ vertices standing before the $(n+1)$-th vertex
in depth-first order. We let $u^{(i)}(0)\prec u^{(i)}(1) \prec \dots$ be the type $i$ vertices of $\mathbf{f}$ arranged
in depth-first order, and we also consider the quantity $G_{i}^{\mathbf{f}}(n) = \# \{u \in \mathbf{f}: u \prec u^{(i)}(n) \}$, 
with the convention $G_{i}^{\mathbf{f}}(\# \mathbf{f}^{(i)}) = \# \mathbf{f}$. Similar 
notation holds if we consider trees instead of forests. Recall that 
$\mathbf{a} = (a_{1}, \dots, a_{d})$ is the left $1$-eigenvector of the mean 
matrix $\mathbf{M}$.

\begin{proposition} \label{pro4}
 For $i \in [d]$ and for any $\mathbf{x} \in [d]^{\mathbb{N}}$, under 
 $\mathbf{P}^{\mathbf{x}}$, we have that
 \begin{eqnarray*}
  \left( \frac{\Lambda_{i}^{F}(\lfloor ns \rfloor)}{n}, s \geq 0 \right) 
  \underset{n \rightarrow \infty}{\rightarrow} \left(a_{i}s, s \geq 0 \right), 
 \end{eqnarray*}

 \noindent in probability, for the topology
 of uniform convergence over compact subsets of $\mathbb{R}_{+}$.
\end{proposition}

\begin{proof}
We only need to prove that for $i \in [d]$, $\varepsilon >0$ and 
for any $\mathbf{x} \in [d]^{\mathbb{N}}$, we have that
  \begin{eqnarray} \label{eq17}
  \mathbf{P}^{\mathbf{x}} \left( \left| G_{i}^{F}(n) - a_{i}^{-1}n \right| > \varepsilon n \right) = 0,
 \end{eqnarray}
 
 \noindent as $n \rightarrow \infty$. This will imply the convergence in 
 probability for every rational number $s$ of $G_{i}^{F}(\lfloor n s \rfloor)n^{-1}$ 
 towards $a_{i}^{-1}s$ as $n \rightarrow \infty$. Then, an application of 
 Skorohod's representation theorem and a standard diagonal procedure entail 
 that the above convergence holds for the uniform topology over compact subsets
 of $\mathbb{R}_{+}$. Finally, one notices that $\Lambda_{i}^{F}$ is the right-continuous inverse function of $G_{i}^{F}$
 which leads to our statement. 

 In this direction, for $\mathbf{f} \in \mathbb{F}^{(d)}$, we recall that
 $\Pi^{(i)}(\mathbf{f})$ denotes the monotype forest obtained after applying the projection function
 described in Section \ref{subs1}. Let $u(0) \prec u(1) \prec \dots$ be the vertices 
 of $\Pi^{(i)}(\mathbf{f})$ listed in depth-first order and recall that for $k \geq 0$ and 
 $j \in [d] \setminus \{i\}$, $N_{ij}(k) : = N_{ij}(u(k))$ denotes the number of type $j$
 vertices that have been deleted between
 $u(k)$ and its children during the operation $\Pi^{(i)}$. 
 Similarly, we define the quantity $N_{ij}^{\prime}(k)$ 
 which counts only the type $j$ vertices that come before $u^{(i)}(n)$ in depth-first 
 order. Since $\sum_{j \neq i}a_{j}/a_{i}= 1- 1/a_{i}$, we notice that
 \begin{eqnarray} \label{eq8}
  G_{i}^{\mathbf{f}}(n) -a_{i}^{-1}n = \sum_{j \neq i} \left( R_{1}^{\mathbf{f}}(j;n) + R_{2}^{\mathbf{f}}(j;n)  + R_{3}^{\mathbf{f}}(j;n)  \right),
 \end{eqnarray}
 
 \noindent for $n \geq 0$ and where for $j \in [d] \setminus \{i\}$,
\begin{eqnarray*}
R_{1}^{\mathbf{f}}(j;n) = \sum_{k=0}^{n-1} \left( N^{\prime}_{ij}(k) - N_{ij}(k) \right) \mathds{1}_{\{ u^{(i)}(k) \vdash u^{(i)}(n) \}}, \hspace{8mm} R_{2}^{\mathbf{f}}(j;n) = \sum_{k=1}^{\Upsilon_{n}^{\mathbf{f}}} \hat{N}_{ij}(k),
\end{eqnarray*} 

\noindent and 
 \begin{eqnarray*}
  R_{3}^{\mathbf{f}}(j;n)  = \sum_{k=0}^{n-1} \left( N_{ij}(k) - a_{j}/a_{i} \right).
 \end{eqnarray*}

\noindent We next estimate the probability that these tree terms is large, when we consider a $d$-type GW forest. We fix 
$ \varepsilon > 0$, $0 < \delta < 1/\underline{\alpha}$ and 
write $z_{n} = n^{1-1/\underline{\alpha}+\delta}$. We observe that
\begin{eqnarray*} 
 \left | R_{1}^{F}(j;n) \right|  \leq \sum_{k=0}^{n-1} N_{ij}(k) \mathds{1}_{\{ u^{(i)}(k) \vdash u^{(i)}(n) \}}.
\end{eqnarray*}

\noindent and
\begin{eqnarray*}
 \# \{ k \geq 0: u^{(i)}(k) \vdash u^{(i)}(n) \} \leq \text{Anc}_{F}^{u^{(i)}(n)}(i) \leq \max_{0 \leq k \leq n} H_{k}^{\Pi^{(i)}(F)}.
\end{eqnarray*}

\noindent Thus, according to our estimate for the height of GW forests in Lemma \ref{lemma2},
we get that
\begin{eqnarray*}
 \mathbf{P}^{\mathbf{x}} \left( \left | R_{1}^{F}(j;n)  \right| > \varepsilon n^{1+\delta}  \right)
 & \leq &  \mathbf{P}^{\mathbf{x}}  \left( \sum_{k=0}^{\lfloor z_{n} \rfloor} N_{ij}(k)  > \varepsilon n^{1+\delta} \right) + \text{oe}(n).
\end{eqnarray*}

\noindent  Moreover, for every $\beta \in (0,1/2)$,
 \begin{align} \label{eq16}
 & \mathbf{P}^{\mathbf{x}} \left( \left | R_{1}^{F}(j; n)  \right| > \varepsilon n^{1+\delta}  \right)  \nonumber \\
 & ~~~~~~~~ \leq  \mathbf{P}^{\mathbf{x}} \left(\left \{  \sum_{k=1}^{ \left\lfloor z_{n} \right\rfloor} N_{ij}(k) > \varepsilon n^{1+\delta}  \right \} \cap \left\{ \forall k \in \{ 0, 1, \dots, \lfloor z_{n} \rfloor \}: N_{ij}(k) < (1-\beta) \varepsilon n^{1+\delta} \right\} \right) \nonumber \\
 & ~~~~~~~~~~~~~~~~  + \mathbf{P}^{\mathbf{x}} \left( \max_{1 \leq k \leq  \lfloor z_{n} \rfloor} N_{ij}(k) > (1-\beta) \varepsilon n^{1+\delta} \right) + \text{oe}(n).
 \end{align}

\noindent We recall that under $\mathbf{P}^{\mathbf{x}}$,
the random variables $(N_{ij}(k), k \geq 0)$ are i.i.d. with law in the domain
of attraction of a stable law of index $\underline{\alpha} \in (1,2]$ by Proposition \ref{pro2} (i). Then, 
\begin{eqnarray*}
  \mathbf{P}^{\mathbf{x}} \left( \max_{0 \leq k \leq  \lfloor z_{n} \rfloor} N_{ij}(k) > (1-\beta) \varepsilon n^{1+\delta} \right)  =  1 - \left( 1- \mathbf{P}^{\mathbf{x}} \left(  N_{ij}(0) > (1-\beta) \varepsilon n^{1+\delta} \right) \right)^{\lfloor z_{n} \rfloor} = 0,
\end{eqnarray*}

\noindent as $n \rightarrow \infty$. On the other hand,
the first term in the right-hand side of (\ref{eq16}) also tends to $0$ as $n \rightarrow \infty$. To see this, note that the event 
in the first term may hold only if there are two distinct 
values of $k \in \{0, 1, \dots, \lfloor z_{n} \rfloor \}$ such that 
$N_{ij}(k) \geq \beta \varepsilon n / \lfloor z_{n} \rfloor$. We thus conclude that
\begin{eqnarray} \label{eq9}
 \mathbf{P}^{\mathbf{x}} \left( \left | R_{1}^{F}(j;n) \right| > \varepsilon n^{1+\delta} \right) = 0,
 \hspace*{6mm} \text{as} \hspace*{5mm} n \rightarrow \infty.
 \end{eqnarray}

\noindent Following exactly the same argument, using the bound in Lemma \ref{lemma2} on 
the number of components of $d$-type GW forests and Proposition \ref{pro2} (ii), 
we obtain that
\begin{eqnarray} \label{eq13}
 \mathbf{P}^{\mathbf{x}} \left( \left| R_{2}^{F}(j;n) \right|> \varepsilon n^{1+\delta} \right) = 0, 
 \hspace*{5mm} \text{as} \hspace*{6mm} n \rightarrow \infty.
\end{eqnarray}

\noindent  Finally, the estimate 
\begin{eqnarray} \label{eq10}
 \mathbf{P}^{\mathbf{x}} \left( \left | R_{3}^{F}(j;n) \right| > \varepsilon n^{1+\delta}  \right) = 0,  \hspace*{5mm} \text{as} \hspace*{5mm} n \rightarrow \infty,
\end{eqnarray}

\noindent follows by the law of large numbers, since Proposition \ref{pro2} (i) entails that 
the mean of $N_{ij}(0)$ is $a_{j}/a_{i}$.  

Therefore, the estimates (\ref{eq9}), (\ref{eq13}) and (\ref{eq10}), when combined 
with (\ref{eq8}) imply the convergence (\ref{eq17}).
\end{proof}

\section{Proof of Theorem \ref{teo1} and \ref{teo2}} \label{pteo1}

In this section, we prove our main results. 

\begin{proof}[Proof of Theorem \ref{teo1}]
We observe that for $n \geq 0$ and any $s\geq 0$, we have
\begin{eqnarray*}
 \left| H_{\lfloor ns \rfloor}^{F} - \frac{H_{\Lambda_{i}^{F}(\lfloor ns \rfloor) -1}^{\Pi^{(i)}(F)}}{a_{i}b_{i}} \right| \leq \left| H_{\lfloor ns \rfloor}^{F} - \frac{\text{Anc}_{F}^{u(\lfloor ns \rfloor)}(i)}{a_{i}b_{i}}  \right| + \frac{1}{a_{i}b_{i}}\left | H_{\Lambda_{i}^{F}(\lfloor ns \rfloor) -1}^{\Pi^{(i)}(F)} - \text{Anc}_{F}^{u(\lfloor ns \rfloor)}(i) \right|.
\end{eqnarray*}

\noindent By Proposition \ref{pro3}, under $\mathbf{P}^{\mathbf{x}}$, the first term on the right hand side tends to $0$ in probability as $n \rightarrow \infty$,  uniformly over compact subsets of $\mathbb{R}_{+}$. On the other hand, from
equation (15) in \cite{Gr}, we get that
\begin{eqnarray*}
 \left | H_{\Lambda_{i}^{F}(\lfloor ns \rfloor) -1}^{\Pi^{(i)}(F)} - \text{Anc}_{F}^{u(\lfloor ns \rfloor)}(i) \right| \leq \left | H_{\Lambda_{i}^{F}(\lfloor ns \rfloor) -1}^{\Pi^{(i)}(F)} - H_{\Lambda_{i}^{F}(\lfloor ns \rfloor)}^{\Pi^{(i)}(F)} \right| +1.
\end{eqnarray*}

\noindent Recall that under $\mathbf{P}^{\mathbf{x}}$, $\Pi^{(i)}(F)$
is a critical non-degenerate monotype GW forest in the domain of attraction of a stable law of index 
$\underline{\alpha} \in (1,2]$ by Proposition \ref{pro1}. Then, Theorem 3.1 in \cite{Du1} implies that
\begin{eqnarray*}
 \frac{1}{n^{1-1/\underline{\alpha}}} \max_{0 \leq k \leq n} \left| H_{k-1}^{\Pi^{(i)}(F)} - H_{k}^{\Pi^{(i)}(F)} \right| \underset{n \rightarrow \infty}{\rightarrow} 0,
 \end{eqnarray*}

\noindent in probability, under $\mathbf{P}^{\mathbf{x}}$, and it follows that 
\begin{eqnarray} \label{eq15}
 \left( \frac{1}{n^{1-1/\underline{\alpha}}} \left( H_{\lfloor ns \rfloor}^{F} - \frac{1}{a_{i}b_{i}} H_{\Lambda_{i}^{F}(\lfloor ns \rfloor)}^{\Pi^{(i)}(F)}  \right), s \geq 0 \right) \underset{n \rightarrow \infty}{\rightarrow} 0
\end{eqnarray}

\noindent in probability for the topology of uniform convergence over compact sets of
$\mathbb{R}_{+}$. Finally, Proposition \ref{pro4} and Theorem 3.1 in \cite{Du1} imply that 
\begin{eqnarray*}
 \left(  \frac{1}{n^{1-1/\underline{\alpha}}} H_{\Lambda_{i}^{F}(\lfloor ns \rfloor)}, s \geq 0 \right) \xrightarrow[n \rightarrow \infty]{d} \left( \frac{a_{i}^{1/\alpha}b_{i}}{\bar{c}} H_{a_{i}s}, s \geq 0 \right).
\end{eqnarray*}

\noindent Moreover, we deduce from the scaling property of the height process $H$ that 
$(H_{a_{i}s}, s \geq 0) \stackrel{d}{=} (a_{i}^{1-1/\underline{\alpha}}H_{s}, s \geq 0)$; see,
e.g., Section 3.1 in \cite{Du}. Therefore, the result in Theorem \ref{teo1} 
follows now from (\ref{eq15}). 
\end{proof}

Let us now prove Theorem \ref{teo2}.

\begin{proof}[Proof of Theorem \ref{teo2}]
For $n \geq 0$, $i \in [d]$ and any $s \geq 0$, we recall that 
$\Lambda_{i}^{F}(\lfloor n s \rfloor)$ denotes the number of type $i$ individuals
standing before the $(\lfloor n s \rfloor+1)$-th individual in depth-first order which
we called $u(\lfloor n s \rfloor)$. Since all the roots of the forest $F$ have type
$i$, we claim that
\begin{eqnarray*}
 \Upsilon_{\Lambda_{i}^{F}(\lfloor n s \rfloor)}^{\Pi^{(i)}(F)} = \Upsilon_{\lfloor n s \rfloor}.
\end{eqnarray*}

\noindent To see this, we observe that $u(\lfloor n s \rfloor)$ and the last vertex of type $i$ 
before $u(\lfloor n s \rfloor)$ in depth-first order belong to the same tree component.
Therefore, the label of the tree component of $F$ containing $u(\lfloor n s \rfloor)$ is
the same as the label of the tree component of $\Pi^{(i)}(F)$ containing the
$\Lambda_{i}^{F}(\lfloor n s \rfloor)$-th vertex. The result now follows
from Proposition \ref{teo1} and similar arguments as in the proof of Theorem \ref{teo1}.
\end{proof}

\section{Applications} \label{app}

\subsection{Maximal height of multitype GW trees}

In this section, we present a natural consequence of Theorems \ref{teo1} and \ref{teo2} which 
generalizes the result of Miermont \cite{Gr} on the maximal height in the finite 
covariance case. For a tree $\mathbf{t} \in \mathbb{T}$, we let $\text{ht}(\mathbf{t})$ be the 
maximal height of a vertex in $\mathbf{t}$. Recall that $I_{s}$ 
is the infimum at time $s$ of the strictly stable spectrally positive 
L\'evy process $Y^{(\underline{\alpha})}$.

\begin{corollary}
For $i \in [d]$, let $T$ be a $d$-type GW tree distributed according to $\mathbf{P}^{(i)}$ whose offspring distribution satisfies the main assumptions. Then, 
\begin{eqnarray*}
\lim_{n \rightarrow \infty} n\mathbf{P}^{(i)} \left({\rm ht}(T) \geq n \right) 
=  b_{i}  (\underline{\alpha}-1) \left( (\underline{\alpha}-1) \bar{c} \right)^{\frac{\underline{\alpha}}{1-\underline{\alpha}}}. \\
\end{eqnarray*}
\end{corollary}

\begin{proof}
The proof of this assertion is very similar of Corollary 1 in \cite{Gr}. The only 
difference that we are now considering that the rescaled height process
of multitype GW forest converges to height process associated with the strictly stable
spectrally positive L\'evy process $Y^{(\underline{\alpha})}$. Let $F$ be a $d$-type GW forest distributed
according to $\mathbf{P}^{(\mathbf{i})}$ whose offspring distribution satisfies the main assumptions. For $k \geq 1$, we denote by $\tau_{k}$ the first hitting time of $k$ by $(\Upsilon^{F}_{n}, n \geq 0)$ and for $x \geq 0$, we write $\varrho_{x}$ for the first hitting time of $x$ by $-I = (-I_{s}, s \geq 0)$. From Theorem \ref{teo1} and \ref{teo2}, we have that 
\begin{eqnarray*}
\left(\frac{1}{n} H_{n^{\frac{\alpha}{\alpha -1}}s}^{F}, 0 \leq s \leq \tau_{n}  \right) \xrightarrow[n \rightarrow \infty]{d} \left(\frac{1}{\bar{c}} H_{s}, 0 \leq s \leq \varrho_{b_{i} \bar{c}^{-1}} \right),
\end{eqnarray*}

\noindent under $\mathbf{P}^{(\mathbf{i})}$. Let $(F_{k}, k \geq 1)$ be the tree
components of the multitype GW forest $F$. Then, the above convergence implies that
\begin{eqnarray*}
\lim_{n \rightarrow \infty} \mathbf{P}^{(\mathbf{i})}\left( \max_{1 \leq k \leq n} \text{ht}(F_{k}) < n \right) & = &\mathbf{P} \left(  H_{s} \leq \bar{c} , \, \text{for all} \, \,  0 \leq s \leq \varrho_{b_{i} \bar{c}^{-1}}\right) \\
& = & \exp \left(  -\frac{b_{i}}{\bar{c_{i}}} N \left (\frac{1}{\bar{c}} \sup H \geq 1 \right)\right) \\
& = & \exp \left( -b_{i}  (\underline{\alpha}-1) \left( (\underline{\alpha}-1) \bar{c} \right)^{\frac{\underline{\alpha}}{1-\underline{\alpha}}} \right),
\end{eqnarray*}

\noindent where $N$ is the It\^o excursion measure of $Y^{(\underline{\alpha})}$ above its infimum
(see e.g. Chapter VIII.2 in \cite{Be} for details), and where we have used the Corollary 
1.4.2 in \cite{Du} for the equality. Recall that 
under $\mathbf{P}^{(\mathbf{i})}$, the tree components $(F_{k}, k \geq 1)$ are independent
multitype GW trees. Therefore, the identity
\begin{eqnarray*}
 \mathbf{P}^{(\mathbf{i})}\left( \max_{1 \leq k \leq n} \text{ht}(F_{k}) < n \right) = \left( 1-  \mathbf{P}^{(i)}\left(  \text{ht}(T) \geq n \right) \right)^{n} .
\end{eqnarray*}

\noindent yields our claim. 
\end{proof}

\subsection{Alternating two-type GW tree}

We consider a particular family of multitype GW trees 
known as alternating two-type GW trees, in which vertices of type $1$ only 
give birth to vertices of type $2$ and vice versa. More precisely, given
two probability measures $\mu_{2}^{(1)}$ and $\mu_{1}^{(2)}$ on $\mathbb{Z}_{+}$, we consider a 
two-type GW tree where every vertex of type $1$ (resp. type 2) has a number of
type 2 (resp. type 1) children distributed according to $\mu_{2}^{(1)}$ (resp. $\mu_{1}^{(2)}$), all independent of each other. 
We denote by ${\bm \mu}_{\text{alt}}$ the offspring distribution on $\mathbb{Z}_{+}^{2}$ 
of this particular two-type GW tree. We let 
\begin{eqnarray*}
 m_{12} = \sum_{z \in \mathbb{Z}_{+}} z \mu_{2}^{(1)}(\{ z \}) 
 \hspace*{6mm} \text{and} \hspace*{6mm} 
 m_{21} = \sum_{z \in \mathbb{Z}_{+}} z \mu_{1}^{(2)}(\{ z \})
\end{eqnarray*}

\noindent be the means of the measures $\mu_{2}^{(1)}$ and $\mu_{1}^{(2)}$, respectively.
We make the assumption that $\mu_{2}^{(1)}(\{ 1 \}) + \mu_{1}^{(2)}(\{ 1 \}) < 2$ to discard
degenerate cases, and also exclude the trivial case $m_{1}m_{2} = 0$. We observe
that the mean matrix associated with ${\bm \mu}_{\text{alt}}$ is irreducible and it admits
$\rho = m_{1}m_{2}$ as a unique positive eigenvalue. We then say that ${\bm \mu}_{\text{alt}}$ 
is sub-critical if $m_{1}m_{2} < 1$, critical if $m_{1}m_{2} = 1$ and supercritical
if $m_{1}m_{2} > 1$. In the sequel, we assume that offspring distribution is also
critical. We observe then that the normalized left and right $1$-eigenvectors are given by
\begin{equation*}
\mathbf{a} = (a_{1}, a_{2})= \left(\frac{1}{1+m_{1}}, \frac{1}{1+m_{2}} \right),
 \hspace*{5mm} \text{and} \hspace*{5mm} 
 \mathbf{b} = (b_{1}, b_{2}) = \left(\frac{1+m_{1}}{2}, \frac{1+m_{2}}{2} \right). 
\end{equation*}

Following the notation of Section \ref{Law}, we denote by $\mathbf{P}_{\text{alt}}^{(i)}$
the law of a two-type GW tree with offspring distribution ${\bm \mu}_{\text{alt}}$ 
and root type $i \in [2]$, i.e., it is the law of an alternating two-type GW tree with root 
type $i$. We make the next extra assumptions on the offspring 
distribution:
\begin{itemize}
 \item[($\mathbf{H}_{1}^{\prime}$)] $\mu_{2}^{(1)}$ is a geometric distribution, i.e. 
there exists $p \in (0,1)$ such that 
\begin{eqnarray*}
 \mu_{2}^{(1)}(\{z\}) = (1-p)p^{z}, \hspace*{6mm} z \in \mathbb{Z}_{+}.
\end{eqnarray*}

\noindent We observe that its Laplace exponent satisfies
 \begin{eqnarray*}
 \psi_{1}(s) = \frac{p}{1-p} s + \frac{1}{2} \frac{p}{(1-p)^{2}} s^{2} +o(s^{2}), \hspace*{6mm} s \downarrow 0,
\end{eqnarray*}
for $s \in \mathbb{R}_{+}$. In particular, $m_{1} = p/(1-p)$.

\item[($\mathbf{H}_{2}^{\prime}$)] $\mu_{1}^{(2)}$ is in the domain of attraction of a stable 
law of index $\alpha \in (1,2]$, that is, its Laplace exponent satisfies  
\begin{eqnarray*}
 \psi_{2}(s) = m_{2} s + s^{\alpha} L(s) + o(s^{\alpha}), \hspace*{6mm} s \downarrow 0,
\end{eqnarray*}
for $s \in \mathbb{R}_{+}$ and where $L: \mathbb{R}_{+} \rightarrow \mathbb{R}_{+}$
is a slowly varying function at zero.
\end{itemize}

The following result is a conditioned version of Theorem \ref{teo1} for
this particular two-type GW tree. More precisely, we show that after a proper 
rescaling the height process of a critical alternating two-type GW tree whose 
offspring distribution satisfies ($\mathbf{H}_{1}^{\prime}$) and ($\mathbf{H}_{2}^{\prime}$)
converges to the normalized excursion of the continuous-time height process 
associated with a strictly stable spectrally positive L\'evy process with index 
$\alpha$. We stress that the improvement of the convergence in
Theorem \ref{teo1} is because we are able to establish a conditioned version of 
Proposition \ref{pro4} for this very particular GW tree. This allows us to adapt
the proof of Theorem 2 in \cite{Gr} without making the extra assumption that the offspring 
distribution has small exponential moments. 

Before providing a rigorous statement, we need to introduce some further 
notation. We consider a function $\bar{L}: \mathbb{R}_{+} \rightarrow \mathbb{R}_{+}$ 
given by
 \begin{eqnarray} \label{eq19}
 \bar{L}(s) = \left( \frac{1}{2} \frac{p}{(1-p)^2} a_{1} b_{2}^{2} \mathds{1}_{\{ \alpha =2 \}} + a_{2} b_{1}^{\alpha} L(s)  \right),
 \hspace*{6mm} \text{for} \, \,  s \in \mathbb{R}_{+},  
 \end{eqnarray}

\noindent which is a slowly varying function at zero. We write $\tilde{L}: \mathbb{R}_{+} \rightarrow \mathbb{R}_{+}$ 
for a slowly varying function at infinity that satisfies
 \begin{eqnarray*} 
 \lim_{s \rightarrow \infty} \left( \frac{1}{\tilde{L}(s)} \right)^{\alpha}\bar{L}\left(\frac{1}{s^{1/\alpha} \tilde{L}(s)} \right) = 1,  
 \end{eqnarray*}

\noindent This function is known in the literature as the conjugate of $\bar{L}$. The existence
of such a function is due to a result of de Bruijn; for a proof of this fact and more 
information about conjugate functions, see Section 1.5.7 in \cite{Bi}. In what follows,
we let $(B_{n}, n\geq 1)$ be a sequence positive integers such that 
$B_{n} = \tilde{L}(n)n^{1/\alpha}$. 

Finally, recall the definition of the discrete height process associated to a 
tree $\mathbf{t} \in \mathbb{T}$; see \cite{Du1} for details and properties.
Let us denote by $\# \mathbf{t}$ the total progeny of $\mathbf{t}$, and $\varnothing = u_{\mathbf{t}}(0) \prec  u_{\mathbf{t}}(1) \prec \dots \prec u_{\mathbf{t}}(\# \mathbf{t} -1)$
be the list of vertices of $\mathbf{t}$ in depth-first order. The height process 
$H^{\mathbf{t}} = (H_{n}^{\mathbf{t}}, n \geq 0)$ is defined by $H_{n}^{\mathbf{t}} = |u_{\mathbf{t}}(n)|$, 
with the convention that $H_{n}^{\mathbf{t}} = 0$ for $n \geq \# \mathbf{t}$.

\begin{theorem} \label{teo3} 
Let $T$ be an alternating two-type GW tree distributed according to 
$\mathbf{P}^{(1)}_{{\rm alt}}$.
Then for $j = 1,2$, under the law $\mathbf{P}^{(1)}_{{\rm alt}}(\cdot | \#T_{j} = n)$, 
the following convergence in distribution holds on 
$\mathbb{D}([0,1], \mathbb{R})$:
  \begin{eqnarray*}
  \left( \frac{B_{n}}{n} H^{T}_{\lfloor \#T s \rfloor}, 0 \leq s \leq 1 \right) 
  \xrightarrow[n \rightarrow \infty]{d} \left( a_{j}^{1/\alpha-1} H_{s}^{\text{exc}}, 0 \leq s \leq 1 \right), 
 \end{eqnarray*}
 \noindent where $H^{\text{exc}}$ is the normalized excursion of the continuous-time
 height process process associated with a 
 strictly stable spectrally positive L\'evy process $Y^{(\alpha)} = (Y_{s}, s \geq 0)$ 
 of index $\alpha$ and with Laplace exponent $\mathbb{E}(\exp(-\lambda Y_{s})) = \exp(-s \lambda^{\alpha})$,
 for $\lambda \in \mathbb{R}_{+}$.
\end{theorem}

In recent years, this special family of two-type GW trees has been the subject
of many studies due to their remarkable relationship with the study of several 
important objects and models of growing relevance in modern probability 
such that random planar maps \cite{Le}, 
percolation on random maps \cite{Ko2}, 
non-crossing partitions \cite{Cy}, to mention just a few. On the other 
hand, up to our knowledge the result of Theorem \ref{teo3} has not been proved before 
under our assumptions on the offspring distribution. Therefore, we believe that this
may open the way to investigate new aspects related to the models mentioned before. 

The proof of Theorem $\ref{teo3}$ relies on some intermediate results. We 
let $T$ be a two-type GW tree with law $\mathbf{P}^{(1)}_{\text{alt}}$. We first 
characterize the law of the reduced forest $\Pi^{(j)}(T)$, for $j = 1,2$. 

\begin{corollary} \label{coro1}
 For $j =1,2$, under the law $\mathbf{P}^{(1)}_{{\rm alt}}$, the tree 
 $\Pi^{(j)}(T)$ is a critical monotype GW forest with non-degenerate offspring
 distribution $\bar{\mu}_{j}$ in the domain of attraction of a stable
 law of index $\alpha$, i.e., its Laplace exponent satisfies that
 \begin{eqnarray*}
  \bar{\psi}_{j}(s) = s +  \frac{1}{a_{j}}\left( \frac{s}{b_{j}}\right)^{\alpha} \bar{L}(s) + o(s^{\alpha}), 
  \hspace{6mm} s \downarrow 0.
 \end{eqnarray*}
 
 \noindent for $s \in \mathbb{R}_{+}$ and where the function $\bar{L}$ is defined in 
 (\ref{eq19}).
\end{corollary}

\begin{proof}
 The results follows from Lemma \ref{lemma1}, after some simple computations. 
\end{proof}

The next step in order to pass from unconditional statements to conditional ones is the 
following estimate for the number of vertices of some specific type in
multitype GW trees. 

\begin{lemma} \label{lemma3}
 Let $T$ be a $d$-type GW tree distributed according to $\mathbf{P}^{(i)}$,
for $i \in [d]$. Then, for every $j \in [d]$: 
\begin{itemize}
 \item[(i)] For some constant $C_{ij} > 0$, we have that
 \begin{eqnarray*}
   \mathbf{P}^{(i)}\left(\# T^{(j)} = n \right)  =  C_{ij} n^{-1 - 1/\underline{\alpha}} + o(n^{-1 - 1/\underline{\alpha}}),
   \hspace*{6mm} \text{as} \, \, n \rightarrow \infty,
 \end{eqnarray*}
\end{itemize}
 \noindent where it is understood that the limit is taken along values 
 for which the probability on the left-hand side is strictly positive.

 \begin{itemize}
  \item[(ii)] The laws of the number of tree components of $\Pi^{(j)}(T)$, under 
  $\mathbf{P}^{(i)}(\cdot | \# T^{(j)}=n)$, converge weakly as $n \rightarrow \infty$. 
 \end{itemize}
\end{lemma}

\begin{proof}
This very similar to Lemma 6 and Lemma 7 in \cite{Gr} and 
the proof is carried out with mild modifications.
\end{proof}
 
 Finally, the last ingredient is a conditioned version of Proposition \ref{pro4} 
 for the alternating two-type GW tree. 

\begin{proposition} \label{pro5}
 For $j = 1,2$, under $\mathbf{P}^{(1)}_{{\rm alt}}(\cdot | \#T^{(j)} = n)$, we have that
 \begin{eqnarray*}
  \left( \frac{\Lambda_{j}^{T}(\lfloor \# T s \rfloor)}{n}, 0 \leq s \leq 1 \right) 
  \underset{n \rightarrow \infty}{\rightarrow} \left(s, 0 \leq s \leq 1 \right), 
 \end{eqnarray*}

 \noindent in probability.
 \end{proposition}
 
 \begin{proof}
 We prove the statement only when $j=1$. The case $j=2$ follows by making occasional 
 changes in the proof below, observing that
 \begin{eqnarray*}
 \Lambda_{1}^{T}( \# T ) + \Lambda_{2}^{T}( \# T ) =  \# T^{(1)} + \# T^{(2)} = \# T.
 \end{eqnarray*}

 We based our proof on a bijection $\mathcal{G}$ due to Janson and Stef\'anson \cite{JaS}
  which maps the alternating two-type GW tree to a standard monotype GW tree. Roughly speaking, 
 this mapping has the property that every vertex of type $1$ is mapped to a leaf, and every
 type $2$ vertex with $k \geq 0$ children is mapped to a vertex with $k+1$ children (the 
 interest reader is refereed to Section $3$ in \cite{JaS}, for details). Moreover, Janson
 and Stef\'anson showed that under $\mathbf{P}^{(1)}_{{\text{alt}}}$,  $\mathcal{G}(T)$
 is a monotype GW tree with offspring distribution given by 
 \begin{eqnarray*}
  \nu(\{0\}) = 1-p, \hspace*{5mm}  \text{and} \hspace*{5mm} 
  \nu(\{z\}) = p \mu_{2} (\{z\}), \hspace*{4mm}
   \text{for} \, \, \, z \in \mathbb{N}.
 \end{eqnarray*}

We notice that $\Lambda_{1}^{T}( \# T ) = \# T^{(1)}$ is exactly the number of leaves of the monotype GW tree $\mathcal{G}(T)$.
Then, Lemma 2.5 in \cite{Ko4} which is a law of large numbers for the number of
leaves of monotype GW trees, implies that for every $\varepsilon > 0$, 
 \begin{eqnarray*}
  \mathbf{P}^{(1)}_{\text{alt}} \left( \sup_{0 \leq s \leq 1} \left| \frac{\Lambda_{1}^{T}(\lfloor \# T s \rfloor)}{\# Ts}- (1-p) \right| > \varepsilon   \Big  | \# T \geq n\right) = \text{oe}(n).
 \end{eqnarray*}

 \noindent We observe that the left $1$-eigenvector $a_{1} = 1-p$. By Lemma \ref{lemma3}, 
 we deduce that 
 \begin{eqnarray} \label{eq23}
  \mathbf{P}^{(1)}_{\text{alt}} \left( \sup_{0 \leq s \leq 1} \left| \frac{\Lambda_{1}^{T}(\lfloor \# T s \rfloor)}{\# Ts}- a_{1} \right| > \varepsilon   \Big  | \# T^{(1)} = n\right) = \text{oe}(n).
 \end{eqnarray}
 
 \noindent Then, if we admit for a while that 
   \begin{eqnarray} \label{eq14}
  \mathbf{P}^{(1)}_{\text{alt}} \left(  \left| \frac{\# T}{n}- \frac{1}{a_{1}} \right| > \varepsilon \Big| \# T^{(1)} = n \right) = \text{oe}(n).
 \end{eqnarray}
 
 \noindent We conclude the proof by combining the above estimate and (\ref{eq23}). 
 
 Let us now turn to the proof of (\ref{eq14}). First, we observe that for $0 < \varepsilon <
 a_{1}^{-1}$, we have that 
  \begin{align} \label{eq18}
  \mathbf{P}^{(1)}_{\text{alt}} \left(  \left| \frac{\# T}{n}- \frac{1}{a_{1}} \right| > \varepsilon, \# T^{(1)} = n \right) & = \mathbf{P}^{(1)}_{\text{alt}} \left(  \# T > \left(\frac{1}{a_{1}} + \varepsilon \right)n , \# T^{(1)} = n \right)  \nonumber \\
 & ~~~~~~~~ + \mathbf{P}^{(1)}_{\text{alt}} \left(  \# T < \left(\frac{1}{a_{1}} - \varepsilon \right)n , \# T^{(1)} = n \right). 
 \end{align}
 
 \noindent The idea is to show that the two term on the right-hand side
 are $\text{oe}(n)$. We start with the first term. We notice that 
 \begin{eqnarray*}
  \mathbf{P}^{(1)}_{\text{alt}} \left(  \# T > \left(\frac{1}{a_{1}} + \varepsilon \right)n , \# T^{(1)} = n \right)
  \leq \sum_{k =n}^{\infty} \mathbf{P}^{(1)}_{\text{alt}} \left(  \# T  = k, \# T^{(1)} < \left(\frac{1}{a_{1}} + \varepsilon \right)^{-1}n \right)
 \end{eqnarray*}
 
 \noindent By recalling that $\# T^{(1)}$ is 
 the number of leaves of the monotype GW tree $\mathcal{G}(T)$,  
 Lemma 2.7 (ii) in \cite{Ko4} implies that terms in the sum are $\text{oe}(n)$. 
 This entails that the first term on the right-hand side of (\ref{eq18}) is $\text{oe}(n)$. 
 We now focus on the second term. We write
 \begin{eqnarray*}
  \mathbf{P}^{(1)}_{\text{alt}} \left(  \# T > \left(\frac{1}{a_{1}} + \varepsilon \right)n , \# T^{(1)} = n \right)
  \leq \sum_{k =n}^{\lfloor (a_{1}^{-1} - \varepsilon)n \rfloor } \mathbf{P}^{(1)}_{\text{alt}} \left(  \# T  = k, \# T^{(1)} > \left(\frac{1}{a_{1}} - \varepsilon \right)^{-1}n \right)
 \end{eqnarray*}
 
 \noindent By using Proposition 1.6, we get that 
  \begin{eqnarray*}
  \mathbf{P}^{(1)}_{\text{alt}} \left(  \# T > \left(\frac{1}{a_{1}} + \varepsilon \right)n , \# T^{(1)} = n \right)
  \leq \sum_{k =n}^{\lfloor (a_{1}^{-1} - \varepsilon)n \rfloor } \frac{1}{n} \mathbf{P}^{(1)}_{\text{alt}} \left( \frac{1}{r} \sum_{r=1}^{k} \mathbf{1}_{\{ X_{r} = -1\}} > \left(\frac{1}{a_{1}} - \varepsilon \right)^{-1} \right),
 \end{eqnarray*}
 
 \noindent where $(X_{r}, r \geq 1)$ is a sequence of i.i.d. random variables 
 with common distribution $\nu (\{\cdot + 1 \})$ on $\{-1\} \cup \mathbb{Z}_{+}$. 
 Then, an application of Lemma 2.2 (i) in \cite{Ko4} shows that this is $\text{oe}(n)$.
 Therefore, we have proved that 
   \begin{eqnarray} \label{eq22}
  \mathbf{P}^{(1)}_{\text{alt}} \left(  \left| \frac{\# T}{n}- \frac{1}{a_{1}} \right| > \varepsilon, \# T^{(1)} = n \right) = \text{oe}(n).
 \end{eqnarray}
 
 \noindent Finally, an appeal to Lemma \ref{lemma3} (i) completes the proof of (\ref{eq14}).
 \end{proof}

We have now all the ingredients to give the proof of  Theorem \ref{teo3}.

\begin{proof}[Proof of Theorem \ref{teo3}]
Recall from Corollary \ref{coro1} that $\Pi^{(j)}(T)$ under $\mathbf{P}^{(1)}_{\text{alt}}$
is a non-degenerate, critical GW forest with offspring distribution
$\bar{\mu}_{j}$ in the domain of attraction of a stable law of index $\alpha \in (1,2]$. 
Thus, by first conditioning on the number of tree components, we obtain using
Lemma \ref{lemma3} (ii) and Theorem 3.1 \cite{Du1} that under 
$\mathbf{P}^{(1)}_{\text{alt}}(\cdot | \# T^{(j)} = n)$,
 \begin{eqnarray*}
   \left( \frac{B_{n}}{n} H^{\Pi^{(j)}(T)}_{\lfloor ns \rfloor}, 0 \leq s \leq 1 \right) 
  \xrightarrow[n \rightarrow \infty]{d} \left( a_{j}^{1/\alpha} b_{j} H_{s}^{\text{exc}}, 0 \leq s \leq 1 \right), 
  \end{eqnarray*}
  
\noindent where the convergence is in distribution on $\mathbb{D}([0,1], \mathbb{R})$.
  To see this, we observe that conditional on the number of tree components
 to be $r$, the GW forest $\Pi^{(j)}(T)$ is composed 
  of $r$ independent GW trees with the same offspring distribution 
  $\bar{\mu}_{j}$. On the other hand, conditioning the sum of their size to 
  be $n$, only one of these trees has size of order $n$, while the other $r-1$ trees
  have total size $o(n)$ with high probability. This implies that the latter
  do not contribute to the limit. We refer to Theorem 5.4 in \cite{Ko3} for details. Then, 
  from Proposition \ref{pro5}, we obtain that under 
  $\mathbf{P}^{(1)}_{\text{alt}}(\cdot | \# T^{(j)} = n)$,
   \begin{eqnarray} \label{eq20}
   \left( \frac{B_{n}}{n} H^{\Pi^{(j)}(T)}_{\Lambda_{j}^{T} (\lfloor \# T s \rfloor)}, 0 \leq s \leq 1 \right) 
  \xrightarrow[n \rightarrow \infty]{d} \left( a_{j}^{1/\alpha} b_{j} H_{s}^{\text{exc}}, 0 \leq s \leq 1 \right), 
  \end{eqnarray}
  
  \noindent in distribution. 
  
On the other hand, recall from the proof of Theorem \ref{teo1}
that for $n \geq 0$ and any $s\geq 0$, we have
\begin{eqnarray} \label{eq21}
 \left| H_{\lfloor \# T s \rfloor}^{T} - \frac{H_{\Lambda_{j}^{T}(\lfloor \# T s \rfloor)}^{\Pi^{(j)}(T)}}{a_{j}b_{j}} \right| \leq \left| H_{\lfloor \# T s \rfloor}^{T} - \frac{\text{Anc}_{T}^{u(\lfloor \# T s \rfloor)}(j)}{a_{j}b_{j}}  \right| +
 R_{n}(s),
 \end{eqnarray}
 
\noindent where 
\begin{eqnarray*}
 |R_{n}(s)| \leq \frac{1}{a_{j}b_{j}} \left(2 \max_{0 \leq k \leq n} \left| H_{k-1}^{\Pi^{(j)}(T)} - H_{k}^{\Pi^{(j)}(T)} \right| +1 \right).
\end{eqnarray*}

\noindent Therefore, it must be clear that our claim follows from the convergence (\ref{eq20})
by providing that the two terms on the right-hand side of (\ref{eq21}) are $o(n/B_{n})$ in probability,
uniformly in $s \in [0,1]$.

 In this direction, we observe from (\ref{eq22}) that 
 $\mathbf{P}^{(1)}_{\text{alt}}(\# T >  \delta n | \# T^{(j)} = n) = \text{oe}(n)$
 for any $\delta > a_{j}^{-1}$. Combining this with Proposition \ref{pro3}, we have
 for $0 < \gamma < \frac{1}{2}(1-1/\alpha)$ and some $C>0$ that
 \begin{align}
  & \mathbf{P}^{(1)}_{\text{alt}}\left( \frac{B_{n}}{n} \max_{0 \leq k \leq \# T} \left| H_{k}^{T} - \frac{\text{Anc}_{T}^{u(k)}(j)}{a_{j}b_{j}}  \right| \geq n^{ -\frac{1}{2}(1-1/\alpha) + \gamma}   \Big| \# T^{(j)}=n\right) \nonumber \\ 
   & ~~~~~~~~ \leq C n^{1+1/\alpha}  \mathbf{P}^{(\mathbf{1})}_{\text{alt}}\left( \frac{B_{n}}{n} \max_{0 \leq k \leq \lfloor \delta n \rfloor} \left| H_{k}^{T} - \frac{\text{Anc}_{T}^{u(k)}(j)}{a_{j}b_{j}}  \right| \geq n^{ -\frac{1}{2}(1-1/\alpha) + \gamma}  \right) + \text{oe}(n) = \text{oe}(n), \nonumber
  \end{align}

 \noindent where $\mathbf{P}^{(\mathbf{1})}_{\text{alt}}$ is the law of alternating two-type 
 GW forest with all its root having type $1$. This show that first term on the right-hand side of (\ref{eq21})
 is $o(n/B_{n})$ in probability, uniformly in $s \in [0,1]$.
 
 Finally, let $\Upsilon^{j}$ be the number of tree components of 
 $\Pi^{(j)}(T)$. Then the law of $\Pi^{(j)}(T)$ under the measure 
 $\mathbf{P}^{(1)}_{\text{alt}} (\cdot | \Upsilon^{j} = r)$ is that of a monotype 
 GW forest with $r$ tree components. Using Theorem 5.4 in \cite{Ko3}, one conclude that 
 for $\varepsilon >0$,
 \begin{eqnarray*}
  \lim_{n \rightarrow \infty} \mathbf{P}^{(1)}_{\text{alt}} \left( \sup_{0 \leq s \leq 1} \frac{B_{n}}{n} |R_{n}(s)| \geq \varepsilon \Big|  \#T^{(j)} = n, \Upsilon^{j} = r \right) = 0.
 \end{eqnarray*}

 \noindent By Lemma \ref{lemma3} (ii), we know that the law of $\Upsilon^{j}$ under 
 $\mathbf{P}^{(1)}_{\text{alt}} ( \cdot |  \#T^{(j)} = n)$ are tight as $n$ varies. Thus, 
  we deduce that the second term on the right-hand side of (\ref{eq21})
 is also $o(n/B_{n})$ in probability, uniformly in $s \in [0,1]$.
\end{proof}

\paragraph{Acknowledgements.} I would like to thank 
Jean Bertoin for several useful discussions and
for his comments on an earlier draft of this manuscript. \\
This work is supported by the Swiss National Science Foundation 200021\_144325/1

 
 \providecommand{\bysame}{\leavevmode\hbox to3em{\hrulefill}\thinspace}
\providecommand{\MR}{\relax\ifhmode\unskip\space\fi MR }
\providecommand{\MRhref}[2]{%
  \href{http://www.ams.org/mathscinet-getitem?mr=#1}{#2}
}
\providecommand{\href}[2]{#2}

\end{document}